\documentclass[11pt, a4paper]{amsart}
\usepackage{amssymb}
\usepackage{graphicx}
\usepackage{color}

\newtheorem{lemm}{Lemma}[section]
\newtheorem{coro}{Corollary}[section]
\newtheorem{prop}{Proposition}[section]
\newtheorem{theo}{Theorem}[section]
\theoremstyle{definition}
\newtheorem{defi}{Definition}[section]
\newtheorem{rema}{Remark}[section]

\numberwithin{equation}{section}

\def\Bbb{\mathbb} 

\def\ep{\varepsilon}

\def\a{\alpha}

\def\fint{\operatorname {--\!\!\!\!\!\int\!\!\!\!\!--}}

\def\rn1{\Bbb R^{N+1}}
\def\rn{\Bbb R^{N}}

\def\a*{\alpha^*_\ep}

\def\R{\mathbb R}
\def\N{\mathbb N}

\begin{document}
\title[Harnack inequality for $p(x)-$type elliptic equations ]{Local bounds, Harnack inequality and H\"older continuity for divergence type elliptic equations with non-standard growth}
\author[Noemi Wolanski]{Noemi  Wolanski}
\address{Departamento  de
Ma\-te\-m\'a\-ti\-ca, Facultad de Ciencias Exactas y Naturales,
Universidad de Buenos Aires, (1428) Buenos Aires, Argentina.}
\email[Noemi
Wolanski]{wolanski@dm.uba.ar}
\thanks{Supported by the Argentine Council of Research CONICET under the project PIP625, Res. 960/12 and  UBACYT 20020100100496.
The authors are members of CONICET}

\maketitle
\begin{abstract}
In this paper we obtain a Harnack type inequality for solutions to elliptic equations in divergence form with non-standard $p(x)-$type growth. A model equation is the inhomogeneous $p(x)-$laplacian. Namely,
\[
\Delta_{p(x)}u:=\mbox{div}\big(|\nabla u|^{p(x)-2}\nabla u\big)=f(x)\quad\mbox{in}\quad\Omega
\]
for which we prove Harnack inequality when $f\in L^{q_0}(\Omega)$ if $\max\{1,\frac N{p_{min}}\}<q_0\le \infty$. The  constant in Harnack inequality depends on $u$ only through  $\||u|^{p(x)}\|_{L^1(\Omega)}^{p_{max}-p_{min}}$. Dependence of the constant on $u$ is known to be necessary in the case of variable $p(x)$. As in previous papers, log-H\"older continuity on the exponent $p(x)$ is assumed.

We also prove that weak solutions are locally bounded and H\"older continuous when $f\in L^{q_0(x)}(\Omega)$ with $q_0\in C(\Omega)$ and $\max\{1,\frac N{p(x)}\}<q_0(x)$ in $\Omega$.

These results are then generalized to elliptic equations
\[
\mbox{div}A(x,u,\nabla u)=B(x,u,\nabla u)
\]
with $p(x)-$type growth.
\end{abstract}

\bigskip

\begin{section}{Introduction}
\label{sect-intro}

The $p(x)$-Laplacian, defined as
\begin{equation*}
\Delta_{p(x)}u:=\mbox{div}(|\nabla u(x)|^{p(x)-2}\nabla u),
\end{equation*}
extends the laplacian, where $p(x)\equiv 2$, and the
$p$-laplacian,
 where $p(x)\equiv p$ with $1<p<\infty$. This operator has been
 used in the modelling of electrorheological fluids (\cite{R}) and
 in image processing (\cite{AMS}, \cite{CLR}), for instance.

 Up to these days, a great deal of results have been obtained for solutions to equations related to this operator. We will only state in this introduction those results that are related to the ones we address in this paper.

 One of the first issues that come into mind is the regularity of solutions to equations involving the $p(x)-$laplacian or more general elliptic equations with $p(x)-$type growth. Another result --that among other things implies  H\"older continuity of solutions-- is Harnack inequality. These two issues have been addressed in several papers and we will describe in this introduction those results we are aware of.

 Let us state, for the record, that our main concern when starting our research was to obtain Harnack inequality for nonnegative  weak solutions of the inhomogeneous equation
 \begin{equation}\label{p(x)-f}
 \Delta_{p(x)}u=f(x)\quad\mbox{in}\quad\Omega
 \end{equation}
 that, strangely enough, had not been addressed previously.

 By a weak solution we mean a function in $W^{1,p(x)}(\Omega)$ that satisfies \eqref{p(x)-f} in the weak sense. (See the definition and some properties of these spaces below).

 When dealing with equations of $p(x)-$type growth it is always assumed that  $1<p_1\le p(x)\le p_2<\infty$ in $\Omega$. Also, some kind of continuity is assumed since most results on $L^p$ spaces cease to hold without any continuity assumption. In particular, in order to get Harnack inequality, log-H\"older continuity is always assumed and we will do so in this paper. (See the definition of log-H\"older continuity below).

 Harnack inequality for solutions of \eqref{p(x)-f} with $f\equiv0$ states that, for any nonnegative bounded weak solution $u$ there exists a constant $C$ --that depends on $u$-- such that, for balls $B_R(x_0)$ such that $B_{4R}(x_0)\subset\Omega$,
 \[
 \sup_{B_R(x_0)}u\le C\big[\inf_{B_R(x_0)}u+R\big].
 \]

 The dependence of $C$ on $u$ cannot be removed as observed with an example in \cite{HKLMP}. In \cite{HKLMP} the authors get this inequality for quasiminimizers of the functional
 \[
 J(u)=\int_\Omega\frac{|\nabla u|^{p(x)}}{p(x)}\,dx.
 \]
 Solutions to \eqref{p(x)-f} with $f\equiv0$ are minimizers, and therefore, quasiminimizers.

 \smallskip

 In \cite{HKLMP} the authors improve the dependence of $C$ on $u$. In fact, in \cite{ZL} Harnack inequality had been obtained with $C$ depending on the $L^\infty$ norm of $u$. In \cite{HKLMP} instead, the dependence was improved to the $L^t$ norm of $u$ for arbitrarily small $t>1$ if $R$ is small enough depending only on $p$ and $t$. In particular, by taking $t=p_1=\inf_\Omega p(x)$ they get a dependence on $\|u^{p(x)}\|_{L^1(B_{4R}(x_0))}$ that is finite by the definition of a weak solution. In particular, no a priori $L^\infty$ bound is involved in Harnack inequality.

 \smallskip

 Later on, the same inequality with a similar dependence on $u$ was obtained for solutions of an obstacle problem related to the functional $J(u)$ in \cite{HHKLM}.

 \smallskip

 We would like to comment that  \cite{ZL} dealt  with a more general equation. Namely,
 \[\Delta_{p(x)}u=(\lambda b(x)-a(x))|u|^{p(x)-2}u\quad\mbox{in}\quad\Omega\]
with $a$ and $b$ nonnegative and bounded and $\lambda$ a positive constant.

 \smallskip

 As is well known, H\"older continuity is deduced form Harnack inequality. Anyway, there are methods that give H\"older continuity for weak solutions without going through Harnack inequality. A result of this kind that applies to more general equations --possibly inhomogeneous-- can be found in \cite{FZ} where the authors prove that bounded weak solutions to
 \begin{equation}\label{eq-general-intro}
 \mbox{div}A(x,u,\nabla u)=B(x,u,\nabla u)\quad\mbox{in}\quad\Omega
 \end{equation}
 are locally H\"older continuous if $A(x,s,\xi), B(x,s,\xi)$ satisfy the structure conditions: For any $M_0>0$ there exist positive constants $\alpha,C_1,C_2,b$ such that, for $x\in\Omega$, $|s|\le M_0$, $\xi\in\R^N$,
 \begin{enumerate}
 \item[(a)] $A(x,s,\xi)\cdot\xi\ge \alpha |\xi|^{p(x)}-b$.

\item[(b)] $\big|A(x,s,\xi)\big|\le C_1|\xi|^{p(x)-1}+b$.

\item[(c)] $\big|B(x,s,\xi)\big|\le C_2|\xi|^{p(x)}+b$.
\end{enumerate}

The condition that $u$ is bounded is essential when the growth of $B$ in the gradient variable is the one in (c). Boundedness is proved in \cite{FZ} under the condition that $B(x,s,\xi)$ grows as $(|s|+|\xi|)^{p(x)-1}$, for instance.

\smallskip

Finally, let us comment that, under additional regularity assumptions on $A$ and $B$ and some different structure conditions (in particular, under the necessary assumption that $p(x)$ be H\"older continuous), H\"older continuity of the derivatives was obtained in \cite{Fan}. (See also \cite{AM} for this result in the case of minimizers of the functional $J(u)$).

\medskip

In the present paper we are mainly concerned with Harnack inequality. Our main goal is to obtain this inequality in the case of an inhomogeneous equation with minimal integrability conditions on the right hand side --that in the case or $p$ constant stand for $f\in L^q(\Omega)$ with $\max\{1,N/p\}< q\le \infty$-- (see the classical paper \cite{S}).

\smallskip

On the other hand, in several applications we found ourselves dealing with families of bounded nonnegative weak solutions --that are not uniformly bounded, not even in $L^{p(x)}-$norm-- and in need of using Harnack inequality with the same constant $C$ for all the functions in the family. As stated above, we could not use any of the known results (not even for solutions of \eqref{p(x)-f} with $f\equiv0$).

In the present paper, a careful follow up of the constants involved in the proofs allows us to see that the dependence of $C$ on $u$ is actually through $\|u^{p(x)}\|_{L^1(B_{4R})}^{p_+^{4R}-p_-^{4R}}$ where $p_+^{4R}=\sup_{B_{4R}}p$ and
$p_-^{4R}=\inf_{B_{4R}}p$. This makes all the difference in many applications.
Anyway, this was also the case in the previous papers on the homogeneous equation. Unfortunately, the results were not stated in this way so that they could not be used in many situations.

\smallskip

We start our paper with the case of \eqref{p(x)-f} in order to show the ideas and techniques in the simplest possible inhomogeneous case. Then, in Section \ref{sect-general-eqs} we consider  weak solutions to \eqref{eq-general-intro} under the structure assumption: For any $M_0>0$ there exist a constant $\alpha$ and nonnegative functions $g_0, C_0\in L^{q_0}(\Omega)$, $g_1,C_1\in L^{q_1}(\Omega)$, $f, C_2\in L^{q_2}(\Omega)$, $K_1\in L^\infty(\Omega)$, $K_2^{p(x)}\in L^{t_2}(\Omega)$ with $\max\{1,\frac N{p_1}\}<q_2,t_2\le\infty$ ($p_1=\inf_\Omega p$), $\max\{1,\frac N{p_1-1}\}<q_0,q_1\le\infty$ such that, for every $x\in\Omega$, $|s|\le M_0$, $\xi\in\R^N$,
\begin{enumerate}
\item $A(x,s,\xi)\cdot\xi\ge \alpha|\xi|^{p(x)}-C_0(x)|s|^{p(x)}-g_0(x)$.

\item $\big|A(x,s,\xi)\big|\le g_1(x)+C_1(x)|s|^{p(x)-1}+K_1(x)|\xi|^{p(x)-1}$.

\item $\big|B(x,s,\xi)\big|\le f(x)+C_2(x)|s|^{p(x)-1}+K_2(x)|\xi|^{p(x)-1}$

\end{enumerate}
and we prove
\begin{theo}\label{harnack-elliptic-intro} Let $\Omega\subset\R^N$ be bounded and let $p$ log-H\"older continuous in $\Omega$.
Let $A(x,s,\xi)$, $B(x,s,\xi)$ satisfy the structure conditions (1), (2) and (3). Let $u\ge0$ be a  bounded weak solution to \eqref{eq-general-intro} and let $M_0$ be such that $u\le M_0$ in $\Omega$. Let $\Omega'\subset\subset\Omega$.  Then, there exist $C$ and $0<R_0\le \min\{1,\frac14\mathrm{dist}(\Omega',\partial\Omega)\}$  such that, for every $x_0\in\Omega'$, $0<R\le R_0$,
\begin{equation}\label{eq-harnack-elliptic}
\sup_{B_{R}(x_0)}u\le C\big[ \inf_{B_{R}(x_0)}u+R+\mu R\big]
\end{equation}
where
\[
\begin{aligned}
\mu&=\Big[R^{1-\frac N{q_2}}\|f\|_{L^{q_2}(B_{4R}(x_0))}\Big]^{\frac1{p_-^{4R}-1}}+\Big[R^{-\frac N{q_0}}\|g_0\|_{L^{q_0}(B_{4R}(x_0))}\Big]^{\frac1{p_-^{4R}-1}}
+\Big[R^{-\frac N{q_1}}\|g_1\|_{L^{q_1}(B_{4R}(x_0))}\Big]^{\frac1{p_-^{4R}-1}}.
\end{aligned}
\]

The constant $C$ depends only on  $\alpha, q_i$,  the log H\"older modulus of continuity of $p$ in $\Omega$, $\mu^{p_+^{4R}-p_-^{4R}}$,  $M^{p_+^{4R}-p_-^{4R}}$, $\|C_i\|_{L^{q_i}(B_{4R}(x_0))}$, $\|K_1^{p(x)}\|_{L^\infty(B_{4R}(x_0))}$, and  $\|K_2^{p(x)}\|_{L^{t_2}(B_{4R}(x_0))}$ where $p_+^{4R}=\sup_{B_{4R}(x_0)}  p$, $p_-^{4R}=\inf_{B_{4R}(x_0)}  p$ and $M=\|u\|_{L^{p_-^{4R}}(\Omega)}$. (Theorem \ref{harnack-elliptic}).

\smallskip

Observe that $\mu^{p_+^{4R}-p_-^{4R}}$ is bounded independently of $R$.

\smallskip

Observe that, when the functions in the structure conditions are independent of $M_0$,  neither $C$ nor $\mu$ depend on the $L^\infty$ norm of $u$. Moreover, in this case any weak solution is locally bounded (see Remark \ref{local bounds-elliptic}).
\end{theo}

%

\medskip

As usual, from Harnack inequality we get H\"older continuity of bounded weak solutions  (Corollary~\ref{smoothness-elliptic}).

\medskip

Let us remark that in this paper we prove that solutions to \eqref{p(x)-f} with $f\in L^{q_0(x)}(\Omega)$ with $q_0\in C(\Omega)$ and $\max\{1,\frac N{p(x)}\}<q_0(x)$ in $\Omega$ are locally bounded (Proposition \ref{local bounds}). In the case of equation \eqref{eq-general-intro}, if the functions in the structure conditions are independent of $M_0$, the local boundedness of weak solutions also holds (see Remark \ref{local bounds-elliptic}).

For solutions of \eqref{p(x)-f} with $f\in L^{q_0(x)}(\Omega)$, with $q_0$ as above,
we also get local H\"older continuity with constant and exponent depending only on the compact subset, $p(x),q_0(x), \||f|^{q_0(x)}\|_{L^1(\Omega)}$ and $\||u|^{p(x)}\|_{L^1(\Omega)}^{p_2-p_1}$ (Corollary \ref{holder-q(x)}).

 With the same ideas, a similar result can be obtained for solutions to \eqref{eq-general-intro} although we do not state this result.

\medskip

On the other hand, if we replace the structure condition (3) by
\begin{enumerate}
\item[(3')] $\big|B(x,s,\xi)\big|\le f(x)+C_2(x)|s|^{p(x)-1}+K_2(x)|\xi|^{p(x)-1}+b|\xi|^{p(x)}$
\end{enumerate}
with $b\in\R_{>0}$, we obtain Harnack inequality for bounded weak solutions (Theorem \ref{harnack-elliptic-alt}). In this case, the constant in Harnack inequality depends also on $bM_0$ where $M_0$ is a bound of $u$.

Again under the structure condition (3'), deduce that if $u$ is a bounded weak solution then, $u$ is locally H\"older continuous (Corollary \ref{smoothness-elliptic-alt}).

\bigskip

Finally, let us observe that even for the simplest homogeneous equation \eqref{p(x)-f} with $f\equiv0$, Harnack inequality does not imply the strong maximum principle which, in the case of $p$ constant states that a nonnegative weak solution that vanishes at a point of a connected set must be identically zero. Therefore, a  proof of this principle that does not make use of Harnack inequality is needed. For the case of $p$ constant, an alternative proof was produced in \cite{V}. We adapt this proof for the variable exponent case in Section 4. We also prove a boundary Hopf lemma.
For the sake of simplicity, we restrict ourselves to the $p(x)-$laplacian.
\section*{\bf Notation and assumptions}

\medskip

Throughout the paper $N$ will denote the spatial dimension.

\subsection*{Assumptions on $p(x)$}
We will assume that the function $p(x)$ verifies
\begin{equation}\label{pminmax}
1<p_{1}\le p(x)\le p_{2}<\infty,\qquad x\in\Omega.
\end{equation}

When we are restricted to a ball $B_r$ we use $p_-^r = p_-(B_r)$ and $p_+^r = p_+(B_r)$ to denote the infimum and the supremum of $p(x)$ over $B_r$.

We also assume that $p(x)$ is continuous up to the boundary and that it has a modulus of continuity $\omega_R:\R \to \R$, i.e.
$|p(x)-p(y)|\leq \omega_R(|x-y|)$ if $x,y\in B_R(x_0)\subset\Omega$ . We will assume that
\[
\omega_R(r)=\frac{C_R}{\big|\log r\big|}\quad\mbox{for}\quad0<r\le1/2
\]
and will refer to such a $\omega_R$ as a log-H\"older modulus of continuity of $p$ in $B_R(x_0)$.

Observe that $p$ log-H\"older continuous implies that
\[
r^{-(p_+^{r}-p_-^r)}\le K_R\quad\mbox{for}\quad0<r\le R
\]
for a constant $K_R$ related to $C_R$.

This fact will be used throughout the paper.

We will say that $p$ is log-H\"older continuous in $\Omega$ if $\omega_R$ is independent of the ball $B_R(x_0)\subset \Omega$.

\subsection*{Definition of weak solution}

Let $1<p_1\le p(x)\le p_2<\infty$ in $\Omega$.

The space $L^{p(x)}(\Omega)$ stands for the set of measurable functions $u$ such that
$|u(x)|^{p(x)}\in L^1(\Omega)$. This is a Banach space with norm
$$
\|u\|_{L^{p(\cdot)}(\Omega)} = \|u\|_{p(\cdot)}  = \inf\Big\{\lambda > 0: \int_\Omega\Big(\frac {|u(x)|}\lambda\Big)^{p(x)}\,dx\leq 1 \Big\}.
$$

The dual space of $L^{p(x)}(\Omega)$ is $L^{p'(x)}(\Omega)$ with $\frac1{p(x)}+\frac1{p'(x)}=1$ for $x\in \Omega$ and duality pairing $\int_\Omega fg\,dx$.

Then, we let $W^{1,p(\cdot)}(\Omega)$ denote the space of measurable functions $u$ such that $u$ and the distributional derivative $\nabla u$ are in $L^{p(\cdot)}(\Omega)$. The norm
$$
\|u\|_{1,p(\cdot)}:= \|u\|_{p(\cdot)} + \| |\nabla u| \|_{p(\cdot)}
$$
makes $W^{1,p(\cdot)}$ a Banach space.

We call $W^{1,p(\cdot)}_0(\Omega)$ the closure in the norm of $W^{1,p(\cdot)}$ of the set of those functions in $W^{1,p(\cdot)}(\Omega$ that have compact support in $\Omega$. When $p$ is log-H\"older continuous, it coincides with the closure of $C_0^\infty(\Omega)$.

Observe that $u\in W^{1,p(\cdot)}$ implies that $|\nabla u|^{p(x)-2}\nabla u\in \big(L^{p'(x)}\big)^N$.

For more definitions and results on these spaces we refer to \cite{DHHR} and \cite{KR}.

\begin{defi} We say that $u$ is a weak solution to \eqref{eq-general-intro} if $u\in W^{1,p(x)}(\Omega)$ and, for every $\phi\in L^\infty(\Omega)\cap W^{1,p(x)}_0(\Omega)$ there holds that
\[
\int A(x,u(x),\nabla u(x))\cdot\nabla \phi(x)\,dx=\int B(x,u(x),\nabla u(x))\phi(x)\,dx.
\]

\end{defi}

%
%
%
%
%
%
%
%


\end{section}

\bigskip

\begin{section}{Harnack inequality for solutions to $\Delta_{p(x)}u=f$}
\label{sect-p-laplace}
 In this section we will prove the following result.
%
%
%

\begin{theo}\label{har} Assume that
$p$ is locally log-H\"older continuous in $\Omega$. Let $x_0\in\Omega$ and $0<R\le1$  is such that $\overline{B_{4R}(x_0)}\subset\Omega$. There exists $C$ such that,
 if  $u$ is a nonnegative  weak solution of the
problem
\begin{equation}
\Delta_{p(x)} u=f \mbox{
in }\Omega,
\end{equation}
with $f\in L^{q_0}(\Omega)$ for some $\max\{1,\frac N{p_-^{4R}}\}<q_0\le\infty$, then,
\begin{equation}\label{har-ineq}
\sup_{{B_{R}}}u\leq C\Big[\inf_{{B_{R}}}u+R+R\mu\Big]
\end{equation}
where
\[
\mu=\big[R^{1-\frac N{q_0}}\|f\|_{L^{q_0}(B_{4R}(x_0))}\big]^{\frac1{p_-^{4R}-1}}.
\]

The constant $C$ depends only on $N$, $p_-^{4R}$,
$p_+^{4R}$, $s$, $q_0$, $\omega_{4R}$,  $\mu^{{p_+^{4R}-p_-^{4R}}}$, $\|u\|_{L^{sq'}(B_{4R}(x_0))}^{p_+^{4R}-p_-^{4R}}$, \linebreak $\|u\|_{L^{sr_0}(B_{4R}(x_0))}^{p_+^{4R}-p_-^{4R}}$ $($for certain $q'=\frac q{q-1}$ with $r_0, q\in(1,\infty)$ and $\frac1{q_0}+\frac1{q}+\frac1{r_0}=1$ depending on $N,q_0$ and $p_-^{4R}$$)$.
Here
$s\ge p_+^{4R}-p_-^{4R}$ is arbitrary and $\omega_{4R}$ is the modulus of  log-H\"older continuity of $p$
in $B_{4R}(x_0)$.
\end{theo}

\medskip

The proof will be a consequence of 3 lemmas

\begin{lemm}[Cacciopoli type estimate]\label{cacciopoli-lemma} Let $u\ge 1$ and bounded such that $\Delta_{p(x)}
u\ge -H(x)u^{p(x)-1}$ in a ball $B$ and $\gamma>0$ or $\Delta_{p(x)}
u\le H(x)u^{p(x)-1}$ in $B$ and $\gamma<0$. Here $ H\ge0$ is a measurable function. Then, for $\eta
\in C_0^\infty(B)$ there holds that,
\begin{equation}\label{cacciopoli}\begin{aligned}
\int_{B}u^{\gamma-1}|\nabla u|^{p_-}\eta^{p_+}&\le\int_{B} u^{\gamma-1}\eta^{p_+}
+C|\gamma|^{-p_+}\int_{B}u^{\gamma+p(x)-1}\eta^{p_+-p(x)}|\nabla\eta|^{p(x)}
\\&+C|\gamma|^{-1}\int_{B} H(x)u^{\gamma+p(x)-1}\eta^{p_+}\end{aligned}
\end{equation}
with $C=C(p_+,p_-)$. Here $p_+=\max_{\overline B}p$, $p_-=\min_{\overline B} p$.
\end{lemm}
\begin{proof}
 As is usual in the proof of these type of estimates we take as a test
function $u^\gamma\eta^{p_+}\in W^{1,p(x)}_0(\Omega)$ since $u\in W^{1,p(x)}(\Omega)$ and we are assuming that $u\in L^\infty(\Omega)$.

Assume first that $\Delta_{p(x)}
u\ge -H(x)u^{p(x)-1}$ and $\gamma>0$. We get,
\[\begin{aligned}
\gamma\int u^{\gamma-1}\eta^{p_+}|\nabla u|^{p(x)}&\le
-p_+\int u^\gamma\eta^{p_+-1}
|\nabla u|^{p(x)-2}\nabla u\cdot\nabla\eta
+\int H(x)u^{\gamma+p(x)-1}\eta^{p_+}\\
&\le
 \ep p_+\int\frac1{p'(x)}
|\nabla u|^{p(x)}u^{\gamma-1}\eta^{p_+}\\&+
\int\frac{p_+}{\ep^{p(x)-1}p(x)}u^{\gamma+p(x)-1}\eta^{p_+-p(x)}|\nabla\eta|^{p(x)}
+\int H(x)u^{\gamma+p(x)-1}\eta^{p_+}
\end{aligned}\]
where $0<\ep\le 1$ is to be chosen and $\frac1{p(x)}+\frac1{p'(x)}=1$.

Now, we choose $\ep=\min\{1,\frac\gamma{2(p_+-1)}\}$ so that $\frac {\ep p_+}{p'(x)}\le \frac \gamma2$, $\frac{p_+}{\ep^{p(x)-1}p(x)}\le C(p_+,p_-)\gamma^{-p_++1}$ and, in order to get
\eqref{cacciopoli} we bound,
\[
\int u^{\gamma-1}\eta^{p_+}|\nabla u|^{p_-}\le \int u^{\gamma-1}\eta^{p_+}
+\int u^{\gamma-1}\eta^{p_+}|\nabla u|^{p(x)}.
\]

Now, if $\Delta_{p(x)}u\le H(x)u^{p(x)-1}$ and $\gamma<0$, since $u\ge 1$ we can proceed as before and we get,
\[\begin{aligned}
\gamma\int u^{\gamma-1}\eta^{p_+}|\nabla u|^{p(x)}&\ge -p_+
\int u^\gamma\eta^{p_+-1}
|\nabla u|^{p(x)-2}\nabla u\cdot\nabla\eta
-\int H(x)u^{\gamma+p(x)-1}\eta^{p_+}
\end{aligned}\]
Dividing by $\gamma$ we get,
\[\begin{aligned}
\int u^{\gamma-1}\eta^{p_+}|\nabla u|^{p(x)}&\le Cp_+|\gamma|^{-p_+}
\int u^\gamma\eta^{p_+-1}
|\nabla u|^{p(x)-2}\nabla u\cdot\nabla\eta
+C|\gamma|^{-1}\int H(x)u^{\gamma+p(x)-1}\eta^{p_+}
\end{aligned}\]

Now the proof continues as before and we obtain \eqref{cacciopoli}.
\end{proof}

\begin{lemm}\label{subhar} Let $p$ log-H\"older continuous in $B_4$. Let $u\ge1$ be bounded and such that $\Delta_{p(x)}u\ge -H(x)u^{p(x)-1}$ in $B_4$ where $0\le H\in L^{q_0}(B_4)$ with $\max\{1,\frac N{p_-^4}\}<q_0\le\infty$. Let  $t>0$. Then, for every $0<\rho_1<\rho_2\le4$ there holds that,
\begin{equation}\label{eq-subhar}
\sup_{B_{\rho_1}}u\le C\Big(\frac{\rho_2}{\rho_2-\rho_1}\Big)^C
\Big(\fint_{B_{\rho_2}}u^t\Big)^{1/t}.
\end{equation}

 The constant $C$  depends  only on  ${s,p_+^4,p_-^4,M^{p_+^4-p_-^4}}$, $\omega_4$, $ \| H(x)\|_{L^{q_0}(B_4)}$, $q_0$, $q$ and $t$. Here $M=\big(\fint_{B_4} u^{sq'}\big)^{sq'}+
 \big(\fint_{B_4} u^{sr_0}\big)^{sr_0}$ with $r_0,q'\in(1,\infty) $ depending on $q_0, p_-^4,N$ and $s\ge p_+^4-p_-^4$ is arbitrary.
\end{lemm}
\begin{proof} We use Moser's iteration technique and we follow the lines of the proof of Lemma 4.6 in \cite{HHKLM} for the treatment of the variable exponent. In our situation we are more careful with the choice of $\kappa$ below in order to get our  result, due to the presence of a right hand side.

In what follows $p_+$ and $p_-$ stand for the maximum and minimum values of
$p$ in $\overline B_\rho$.

Let $0<\sigma<\rho\le 4$.
 Let $\eta\in C_0^\infty(B_\rho)$ such that $\eta\equiv1$ in $B_\sigma$ and
$|\nabla\eta|\le C\frac1{\rho-\sigma}$.

Let $\kappa=\frac {\hat N}{\hat N-p_-^4}$ with $\hat N=N$ when $N>p_-^4$  and,  $p_-^4<\hat N<q_0p_-^4$ when $N\le p_-^4$.

Then, for $\gamma\ge \gamma_0>0$ using \eqref{cacciopoli}, Sobolev inequality and the fact that $\kappa p_-\le p_-^{*}=\frac{Np_-}{N-p_-}$ when $N>p_-^4$, $u\in W^{1,p_-}(B_\rho)$ and, $W_0^{1,p_-}(B_\rho)\subset L^t(B_\rho)$ continuously for every $1<t<\infty$ when $N\le p_-^4\le p_-$,
\[\begin{aligned}
\Big(&\fint\big(u^{\frac{\gamma-1+p_-}{p_-}}\eta^{p_+/p_-}\big)^{\kappa p_-}\Big)
^{1/\kappa p_-}\le
C\rho \Big(\fint\big|\nabla\big(u^{\frac{\gamma-1+p_-}{p_-}}
\eta^{p_+/p_-}\big)\big|^{ p_-}\Big)
^{1/ p_-}\\
&\le C\frac{\gamma-1+p_-}{p_-}\rho\Big(\fint u^{\gamma-1}
\eta^{p_+}|\nabla u|^{p_-}\Big)^{1/ p_-}
+C\rho\frac{p_+}{p_-}\Big(\fint u^{\gamma-1+p_-}
\eta^{p_+-p_-}|\nabla\eta|^{p_-}\Big)^{1/ p_-}\\
&\le C\rho(1+\gamma)\left[\Big(\fint u^{\gamma-1}\eta^{p_+}\Big)^{1/p_-}+
\Big(\fint u^{\gamma-1+p(x)}\eta^{p_+-p(x)}|\nabla\eta|^{p(x)}\Big)^{1/p_-}\right.\\
&\left.+\Big(\fint H(x)u^{\gamma-1+p(x)}\eta^{p_+}\Big)^{1/p_-}\right]+C\rho\Big(\fint
u^{\gamma-1+p_-}\eta^{p_+-p_-}|\nabla\eta|^{p_-}\Big)^{1/p_-}.
\end{aligned}
\]
Here the constant $C$ depends on $p_+^4,p_-^4$ and $\gamma_0$. 

Since, by the choice of $\hat N$ there holds that $q_0>\frac {\hat N}{p_-^4}$, there exists $1<q<\frac {\hat N}{\hat N-p_-^4\mathrm{}}$ such that
$\frac1 q+\frac1{q_0}<1$. Let $r_0\in(1,\infty)$ given by $\frac1 q+\frac 1{q_0}+\frac 1{r_0}=1$.
Now we bound,
\[
\fint u^{\gamma-1}\eta^{p_+}\le \fint u^{\gamma-1+p_-}\eta^{p_+}\le \Big(\fint u^{(\gamma-1+p_-)q}\eta^{qp_+}\Big)^{1/q}
\le C\Big(\frac1{\rho-\sigma}\Big)^{p_+}\Big(\fint_{B_\rho} u^{(\gamma-1+p_-)q}\Big)^{1/q}
 \]
since $\eta\le 1\le \frac4{(\rho-\sigma)}$. And, with $M_1=\Big(\fint_{B_4}u^{sq'}\Big)^{1/sq'}$, $q'=\frac q{q-1}$, $s\ge p_+-p_-$,
\[
\begin{aligned}
&\fint u^{\gamma-1+p(x)}\eta^{p_+-p(x)}|\nabla\eta|^{p(x)}\le
C\Big(\frac1{\rho-\sigma}\Big)^{p_+}\fint_{B_\rho}  u^{\gamma-1+p-}u^{p(x)-p_-}\\
&\le C\Big(\frac1{\rho-\sigma}\Big)^{p_+}\Big(\fint_{B_\rho}  u^{(\gamma-1+p_-)q}\Big)^{1/q}
\Big(\fint_{B_\rho}  u^{(p(x)-p_-)q'}\big)^{1/q'}\\
&\le C\Big(\frac1{\rho-\sigma}\Big)^{p_+} M_1^{p_+-p_-}\Big(\fint_{B_\rho}  u^{(\gamma-1+p_-)q}\Big)^{1/q}.
\end{aligned}
\]
Similarly,
\[\fint
u^{\gamma-1+p_-}\eta^{p_+-p_-}|\nabla\eta|^{p_-}\le
C\Big(\frac1{\rho-\sigma}\Big)^{p_+}\Big(\fint_{B_\rho}  u^{(\gamma-1+p_-)q}\Big)^{1/q}.
\]
Finally, with $M_2=\big(\fint_{B_4}u^{sr_0}\big)^{1/sr_0}$, $s\ge p_+-p_-$,
\[
\begin{aligned}
&\fint H(x)u^{\gamma-1+p(x)}\eta^{p_+}\le  \fint_{B_\rho}  H(x)u^{\gamma-1+p_-}u^{p_+-p_-} \\ &\le\Big(\fint_{B_\rho}  H(x)^{q_0}\Big)^{1/q_0}\Big(\fint_{B_\rho}  u^{(\gamma-1+p_-)q}\Big)^{1/q}\Big(\fint_{B_\rho}
u^{r_0(p_+-p_-)}\Big)^{1/r_0}\\
&\le CM_2^{p_+-p_-}\Big(\frac1{\rho-\sigma}\Big)^{p_+} \Big(\fint_{B_\rho}  u^{(\gamma-1+p_-)q}\Big)^{1/q}
\end{aligned}
\]
with $C$ depending on $q_0,p_+, p_-$  and $ \|H\|_{L^{q_0}(B_4)}$. In fact,
$\rho^{-\frac{N}{q_0}}\le C\rho^{-p_-^4}\le C\rho^{-p_-}\le C\rho^{-p_+}\le (\rho-\sigma)^{-p_+}$.

  Since $M=M_1+M_2\ge 1$  we conclude that,
  \[
\begin{aligned}
\Big(&\fint\big(u^{\frac{\gamma-1+p_-}{p_-}}\eta^{p_+/p_-}\big)^{\kappa p_-}\Big)
^{1/\kappa p_-}&\le C\rho(1+\gamma)\frac{M^{\frac{p_+}{p_-}-1}}{(\rho-\sigma)^{
p_+/p_-}} \Big(\fint_{B_\rho} u^{(\gamma-1+p_-)q}\Big)^{1/qp_-}
\end{aligned}
\]
with $C$ depending on $q_0,p_+^4,p_-^4$,  $ \| H(x)\|_{L^{q_0}(B_4)}$ and $\gamma_0$.

Let us now take $\beta>p_--1$. Then, $\beta=\gamma-1+p_-$ with $\gamma=\beta-(p_--1)>0$. Recalling that
$\rho^{p_-}\le C\rho^{p_+}$ for a constant $C$ that depends only on the log-H\"older
continuity of $p$,
\begin{equation}\label{beta}
\Big(\fint_{B_\sigma}u^{\kappa\beta}\Big)^{1/\kappa}\le C\Big(\frac\rho\sigma\Big)
^{\frac N\kappa}M^{p_+-p_-}\Big(\frac{\rho}{\rho-\sigma}\Big)^{p_+}(1+\beta)^{p_-}
\Big(\fint_{B_\rho}u^{q\beta}\Big)^{1/q}.
\end{equation}

 Let us call
\[
\phi(f,t,E):=\Big(\fint_E|f|^t\Big)^{1/t}.
\]
Then, if $\beta>p_--1$, $s\ge p_+-p_-$, we have for a constant $C$ depending on $p_+^4,p_-^4$,  $ \| H(x)\|_{L^{q_0}(B_4)}$ and $\gamma_0>0$ such that $\beta-(p_--1)\ge \gamma_0$,
\[
\phi(u,\kappa\beta, B_\sigma)\le C^{1/\beta}M^{\frac{p_+-p_-}\beta}(1+\beta)^{p_-/\beta}
\Big(\frac\rho\sigma\Big)^{\frac N{\kappa\beta}}
\Big(\frac\rho{\rho-\sigma}\Big)^{p_+/\beta}\phi(u,q\beta,B_\rho).
\]

And we have a result quite similar to Lemma 4.6 in \cite{HHKLM}. For the sake of completeness we finish the proof.

To this end, we write $\kappa\beta=\bar\kappa \bar\beta$ with $\bar\kappa=\frac\kappa q$ and $\bar\beta=q\beta$.
Recall that, due to the choice of $q$, we have $q<\kappa$. So that, $\bar\kappa>1$ and
\begin{equation}\label{ineq-iterative}
\phi(u,\bar\kappa\bar\beta, B_\sigma)\le C^{q/\bar\beta}M^{\frac{q(p_+-p_-)}{\bar\beta}}(1+\bar\beta)^{qp_-/\bar\beta}
\Big(\frac\rho\sigma\Big)^{\frac N{\bar\kappa\bar\beta}}
\Big(\frac\rho{\rho-\sigma}\Big)^{qp_+/\bar\beta}\phi(u,\bar\beta,B_\rho).
\end{equation}

Let $0<\rho_1<\rho_2\le4$ and let us call $r_j=\rho_1+2^{-j}(\rho_2-\rho_1)$. We will consider \eqref{ineq-iterative} with $\sigma=r_{j+1}$ and $\rho=r_j$. Observe that,
\[\frac\rho\sigma=\frac{r_j}{r_{j+1}}\le 2,\quad\frac\rho{\rho-\sigma}=\frac{r_j}{r_j-r_{j+1}}=
\frac{\rho_1+2^{-j}(\rho_2-\rho_1)}{2^{-(j+1)}(\rho_2-\rho_1)}\le 2^{j+1}\frac{\rho_2}{\rho_2-\rho_1}.
\]
Assume first that $t>q(p_+^4-1)$. Take $\bar\beta_j=\bar\kappa^j t$. There holds that $\bar\beta_j=q\beta_j$ with $\beta_j=\bar\kappa^j\frac tq$. And, $\gamma_j=\beta_j-(p_-^{r_j}-1)\ge \frac tq-(p_+^4-1)=\gamma_0>0$.

 Then, the constant $C$ in every step of the iteration may be taken depending on $\gamma_0$ and independent of $j$.  Thus, we have with $C_0$ depending on ${p_+^4.p_-^4,M^{p_+^4-p_-^4}}$, $\omega_4$, $ \| H(x)\|_{L^{q_0}(B_4)}$, $q_0$ and $t$,
\[\begin{aligned}
\phi&(u,\bar\kappa^{j+1}t, B_{r_{j+1}})\\
&\le C^{qt^{-1}\bar\kappa^{-j}}M^{\frac{q(p_+^4-p_-^4)}{t\bar\kappa^{j}}}
(1+\bar\kappa^jt)^{\bar\kappa^{-j}qt^{-1}p_+^4}\Big(\frac{r_j}
{r_{j+1}}\Big)^{N t^{-1}\bar\kappa^{-(j+1)}}\Big(\frac{r_j}{r_j-r_{j+1}}\Big)^{qp_+^4t^{-1}
\bar\kappa^{-j}}
\phi(u,\bar\kappa^jt, B_{r_j})\\
&\le C_0^{\bar\kappa^{-j}}(1+\bar\kappa^jt)^{\bar\kappa^{-j}qt^{-1}p_+^4}
\big(2^{j+1}\frac{\rho_2}{\rho_2-\rho_1}\Big)
^{qp_+^4t^{-1}\bar\kappa^{-j}}\phi(u,\bar\kappa^jt, B_{r_j}).
\end{aligned}
\]
Iterating this inequality we get,
\[
\begin{aligned}
\phi&(u,\bar\kappa^{j+1}t, B_{r_{j+1}})\\
&\le C_0^{\sum_{i=0}^j\bar\kappa^{-i}}
\Big(\prod_{i=0}^j(1+t\bar\kappa^i)^{t^{-1}\bar\kappa^{-i}}\Big)^{qp_+^4}
\Big(\frac{\rho_2}{\rho_2-\rho_1}\Big)^{qp_+^4t^{-1}\sum_{i=0}^j\bar\kappa^{-i}} \big(2^{qp_+t^{-1}}\big)^{\sum_{i=0}^j(i+1)\bar\kappa^{-i}}\phi(u,t,B_{\rho_2}).
\end{aligned}
\]

Letting $j\to\infty$,
\[
\sup_{B_{\rho_1}}u\le C_0^{\frac1{1-\bar\kappa^{-1}}}
\Big(\prod_{i=0}^\infty(1+t\bar\kappa^i)^{t^{-1}\bar\kappa^{-i}}\Big)^{qp_+^4}
\big(2^{qp_+^4t^{-1}}\big)^{\sum_{i=0}^\infty(i+1)\bar\kappa^{-i}}
\Big(\frac{\rho_2}{\rho_2-\rho_1}\Big)^{qp_+t^{-1}\frac1{1-\bar\kappa^{-1}}}
\Big(\fint_{B_{\rho_2}}u^t\Big)^{1/t}
\]
and the lemma  is proved for $t>q(p_+^4-1)$ since $\prod_{i=0}^\infty(1+t\bar\kappa^i)^{t^{-1}\bar\kappa^{-i}}\le C$.

In order to get the result for $0<t\le q(p_+^4-1)$ we proceed again as in \cite{HHKLM} and use the extrapolation result Lemma 3.38 in \cite{HKM} with $s=\infty$, $p>q(p_+^4-1)$ fixed (here $q$ is the one in our paper, $s$ and $p$ the ones in \cite{HKM}) and $q=t$ (here $q$ is the one in \cite{HKM} and not the one in our paper) that we state below.
\end{proof}

\begin{lemm}[Lemma 3.38 in \cite{HKM}]\label{HKM} Suppose that $0<q<p<s\le\infty$, $\xi\in\R$, and that $B=B_r(x_0)$ is a ball. If a nonnegative function $v\in L^p(B)$ satisfies
\[
\Big(\fint_{\lambda B'}v^s\,dx\Big)^{1/s}\le c_1(1-\lambda)^\xi\Big(\fint_{B'}v^p\,dx\Big)^{1/p}
\]
for each ball $B'=B(x_0,r')$ with $r'\le r$ and for all $0\le\lambda<1$, then
\[
\Big(\fint_{\lambda B}v^s\,dx\Big)^{1/s}\le c(1-\lambda)^{\xi/\theta}\Big(\fint_{B}v^q\,dx\Big)^{1/q}
\]
for all $0\le \lambda<1$. Here $c=c(p,q,s,\xi,c_1)$  and $\theta\in(0,1)$ such that
\[
\frac1p=\frac\theta q+\frac{1-\theta}s.
\]
\end{lemm}

\begin{rema}\label{rem-t0} Observe that it is enough to prove Lemma \ref{subhar} for $t\ge t_0>0$ with $t_0$ arbitrary depending only on $p_+^4,p_-^4,q$ and then, use Lemma \ref{HKM} in order to get the result for $0<t<t_0$. This means that, in order to prove Lemma \ref{subhar}, it is enough to get \eqref{ineq-iterative} for $\bar\beta\ge q\beta_0$ with, for instance, $\beta_0\ge 2(p_+^4-1)$ (this means to have $\gamma_0\ge p_+^4-1$).
\end{rema}

%

Now, we prove a weak Harnack inequality for supersolutions. There holds,

\begin{lemm}[Weak Harnack inequality]\label{superhar}  Let $p$ log-H\"older continuous in $B_4$.  Let $0\le H\in L^{q_0}(B_4)$ with  $\max\{1,\frac N{p_-^{4}}\}<q_0\le\infty$ and let $s\ge p_+^4-p_-^4$. There exists $t_0>0$  depending only on $s,p_-^4,p_+^4$,  $ \| H(x)\|_{L^{q_0}(B_4)}$, $\omega_4$ and  $M^{p_+^4-p_-^4}$ with $M=\big(\fint_{B_4} u^{sq'}\big)^{sq'}+
 \big(\fint_{B_4} u^{sr_0}\big)^{sr_0}$ for some choice of $1<q'=\frac q{q-1}<\infty$ depending on $N, p_-^4,q_0$, $1<r_0<\infty$ with $\frac1{q_0}+\frac1q+\frac1{r_0}=1$;  $C>0$  depending on the same constants and also on $t_0,q_0,q$  such that, for $u\ge1$ and bounded with $\Delta_{p(x)}u\le H(x)u^{p(x)-1}$ in $B_4$  there holds that,
\[\inf_{B_{{1}}}u\ge C\Big(\fint_{B_{2}}u^{t_0}\Big)^{1/t_0}.
\]

\end{lemm}
\begin{proof}The proof follows the lines of the one of Lemma \ref{subhar}. This time we use Cacciopoli inequality \eqref{cacciopoli} with $\gamma< -\gamma_0=-(p_-^4-1)<0$.
We call again  $\kappa=\frac {\hat N}{\hat N-p_-^4}$ with $\hat N$ as in the proof of Lemma \ref{subhar} and choose $q$ and $r_0$ as in that Lemma. Then, we take $0< \sigma<\rho\le 4$. For $\beta=\gamma+(p_--1)<0$ we  prove that,
\begin{equation}\label{beta2}
\phi(u,q\beta,B_\rho)\le C^{1/|\beta|}(1+|\beta|)^{p_+/|\beta|}
\Big(\frac\rho{\rho-\sigma}\Big)^{p_+/|\beta|}\phi(u,\kappa\beta,B_\sigma).
\end{equation}
Here $C$ is a constant depending on $s,q_0,q, p_+^4,p_-^4,\gamma_0=p_-^4-1$, $ \| H(x)\|_{L^{q_0}(B_4)}$ and $M^{p_+^4-p_-^4}$.

In fact, we proceed as in the proof of Lemma \ref{subhar} until we get \eqref{beta}. Then, since $\beta<0$ we get \eqref{beta2}.

Observe that \eqref{beta2} holds for any $\beta<0$ since this is equivalent to $\gamma<-(p_--1)\le -(p_-^4-1)$.

In order to finish the proof it is necessary to prove that there exists $t_0>0$ and $\bar C>0$ depending only on $p_+^4,p_-^4$,  $ \| H(x)\|_{L^{q_0}(B_4)}$, $M^{p_+^4-p_-^4}$ and the log-H\"older modulus of continuity of $p$ in $B_4$ such that,
\begin{equation}\label{reverse}
\phi(u,t_0,B_{2})\le \bar C\phi(u,-t_0,B_{2}).
\end{equation}
Then,  we choose $\beta=-\frac{t_0}{q}$ in \eqref{beta2} in order to start the iterative process.

In order to prove \eqref{reverse}, we let $0<r\le2$  and we bound by using Cacciopoli inequality \eqref{cacciopoli} with $\gamma=1-p_-^{2r}$, $\eta\in C_0^\infty(B_{2r})$ with $\eta\equiv1$ in $B_r$, $|\nabla \eta|\le \frac Cr$,
\[\begin{aligned}
\fint_{B_r} |\nabla \log u|^{p_-^{2r}}&=\fint_{B_r}u^{-p_-^{2r}}|\nabla u|^{p_-^{2r}}\le
C\fint_{B_{2r}}u^{-p_-^{2r}}\eta^{p_+^{2r}}|\nabla u|^{p_-^{2r}}\\
&\le C\fint_{B_{2r}}u^{-p_-^{2r}}\eta^{p_+^{2r}}+\frac C{(p_-^{2r}-1)^{p_+^{2r}}}\fint_{B_{2r}}u^{p(x)-p_-^{2r}}\eta^{p_+^{2r}-p(x)}|\nabla\eta|^{p(x)}\\
&+\frac C{p_-^{2r}-1}\fint_{B_{2r}}H(x)u^{p(x)-p_-^{2r}}\eta^{p_+^{2r}-p(x)}\\
&\le C(p_+^4,p_-^{4})\big[1+r^{-p_+^{2r}}M_1^{p_+^4-p_-^4} \big]+\frac C{p_-^{2r}-1}\fint_{B_{2r}}H(x)u^{p(x)-p_-^{2r}}.
\end{aligned}
\]

The last term can be bound in the following way
\[\begin{aligned}
\fint_{B_{2r}}H(x)u^{p(x)-p_-^{2r}}&\le \Big(\fint_{B_{2r}}H^{q_0}\Big)^{1/q_0}
\Big(\fint_{B_{2r}}u^{(p_+^{2r}-p_-^{2r}){q_0'}}\Big)^{1/{q_0'}}\\
&\le Cr^{-N/q_0}\|H\|_{L^{q_0}(B_4)}\Big(\fint_{B_{2r}}u^{(p_+^{2r}-p_-^{2r}){r_0}}\Big)^{1/r_0}\\
&\le Cr^{-p_+^{2r}}\|H\|_{L^{q_0}(B_4)}M_2^{p_+^{2r}-p_-^{2r}}
\end{aligned}
\]
since $q_0'\le r_0$, $\frac N{q_0}< p_-^4\le p_-^{2r}\le p_+^{2r}$, $0<r\le 2$,.

%
%


 Gathering all these estimates we get,
\[
\fint_{B_r} |\nabla \log u|^{p_-^{2r}}\le C(p_+^4,p_-^{4},\|H\|_{L^{q_0}(B_4)},\omega_{4})\ r^{-p_+^{2r}}M^{p_+^4-p_-^4}.
\]

Now the proof follows in a standard way. By the Poincar\'e inequality applied to $f=\log u$,  using that $r^{p_-^{2r}}\le C r^{p_+^{2r}}$,
\[\begin{aligned}
\fint_{B_r}|f-f_{B_r}|^{p_-^{2r}}&\le C r^{p_-^{2r}}\fint_{B_r}|\nabla f|^{p_-^{2r}}
\le C(p_+^4,p_-^{4},\|H\|_{L^{q_0}(B_4)},\omega_4)M^{p_+^4-p_-^4}
\end{aligned}\]

 Since this holds for every ball $B_r$ with $r\le 2$, by the John-Nirenberg Lemma there exist constants $C_1$ and $C_2$ depending only on $p_-^4,p_+^4,\|H\|_{L^{q_0}(B_4)},\omega_4$ and $M^{p_+^4-p_-^4}$ such that
\[
\fint_{B_{2}}e^{C_1|f-f_{B_2}|}\le C_2
\]
where $f_{B_2}=\fint_{B_2}f$.

We conclude that,
\[\begin{aligned}
\Big(\fint_{B_ {2}}e^{C_1f}\Big)\Big(\fint_{B_ {2}}e^{-C_1f}\Big)&=\Big(\fint_{B_ {2}}e^{C_1(f-f_{B_ {2}})}\Big)
\Big(\fint_{B_ {2}}e^{-C_1(f-f_{B_ {2}})}\Big)\\
&\le \Big(\fint_{B_ {2}}e^{C_1|f-f_{B_ {2}}|}\Big)^2\le C_2^2
\end{aligned}
\]
and we have \eqref{reverse} with $t_0=C_1$.

Now the proof of the lemma ends by an iterative process similar to the one in Lemma \ref{subhar}. In fact,  we call $\bar\kappa=\frac\kappa {q}$, $\bar \beta= q\beta$ and, for the iteration we let $\bar\beta_j=-\bar\kappa^j t_0$, $r_j=1+2^{-j}$. Hence, $\gamma_j=\beta_j-(p_-^{r_j}-1)=-\bar\kappa^j\frac{t_0}{q}-(p_-^{r_j}-1)\le -\gamma_0:=-(p_-^4-1)$. Then, with $\bar C$ the constant in \eqref{reverse}, using that $p_-^{r_j}, p_+^{r_j}\le p_+^4$,
\[
\begin{aligned}
\bar C^{-1}&\phi(u,t_0,B_ {2})\le \phi(u,-t_0,B_ {2})\\
&\le
C_0^{\sum_{i=0}^j\bar\kappa^{-i}}\Big(\prod_{i=0}^j(1+t_0\bar\kappa^i)^{t^{-1}
\kappa^{-i}}\Big)^{q_0p_+^4}
\big(2^{qp_+^4t_0^{-1}}\big)^{\sum_{i=0}^j(i+2)\bar\kappa^{-i}}
\phi(u,-\bar\kappa^{j+1}t_0,B_{r_{j+1}}).
\end{aligned}
\]
Thus,
\[
\Big(\fint_{B_{2}}u^{t_0}\Big)^{1/t_0}\le C\lim_{j\to\infty}\phi(u,-\kappa^{j}t_0,B_{r_j}).
=C\inf_{B_{1}}u.
\]
And the lemma is proved.
\end{proof}

\medskip

We can improve on Lemma \ref{superhar} in the following way (see \cite{MZ} where this improvement was done in the case of $p$ constant)
\begin{lemm}[Improved weak Harnack inequality]\label{superhar-improved}
Under the assumptions of Lemma \ref{superhar}, let $0<t< \frac{N}{N-p_-^4}(p_-^4-1)$ if $p_-^4<N$, $t>0$ arbitrary if $p_-^4\ge N$. Then, there exists a constant $C$ with the same dependence as the constant in Lemma \ref{superhar} and also depending on $t$ such that
\[
\Big(\fint_{B_2} u^t\Big)^{1/t}\le C\inf_{B_1} u
\]
\end{lemm}
\begin{proof} We prove that, for every $t$ in this range, $t_0$ the one in Lemma \ref{superhar}, $0<\rho_1<\rho_2\le 4$ there holds that
\begin{equation}\label{t-bound}
\Big(\fint_{B_{\rho_1}} u^t\Big)^{1/t}\le \bar C \Big(\fint_{B_{\rho_2}} u^{ t_0}\Big)^{1/ t_0}
\end{equation}
for a constant $\bar C$ depending on $t,t_0,\rho_1,\rho_2, M^{p_+^4-p_-^4},p_+^4,p_-^4, q_0$.

 This will prove the lemma if we replace in the proof of Lemma \ref{superhar} the ball $B_2$ by $B_{\rho_2}$ with $2<\rho_2<4$ and we take $\rho_1=2$ in \eqref{t-bound}.

 In order to prove \eqref{t-bound}, we proceed as in Lemma \ref{superhar} but we are more careful with the choice of $\kappa$. In fact, as in Lemma~\ref{superhar} we choose $\kappa=\frac{\hat N}{\hat N-p_-^4}$ with $\hat N=N$ if $p_-^4<N$ and $p_-^4<\hat N<q_0 p_-^4$ if $p_-^4\ge N$. In this latter case, we choose $\hat N$ close enough to $p_-^4$ so that $\kappa^{-1} t=  t\big(1-\frac{p_-^4}{\hat N}\big)<p_-^4-1 $.

 Observe that  $\kappa^{-1}  t<p_-^4-1$ also  if $p_-^4<N$. In fact, in this case we have $\kappa=\frac N{N-p_-^4}$ and  the inequality holds  by our hypothesis on $t$.

 Then, we choose $q$ as in Lemma \ref{superhar}. This is, $1\le q_0'<q<\kappa$.

  In order to prove \eqref{t-bound} we go back to \eqref{beta}. Recall that we get this inequality if $\gamma\le-\gamma_0<0$ and $\beta=\gamma+p_--1$.

 Then,  as in Lemma \ref{superhar}, we take $\bar \beta=q\beta$, $\bar\kappa=\frac\kappa q>1$.

Now, for $j\in\N$ and $i=0,1,\cdots,j$ we let $\bar \beta_{ij}=\bar\kappa^{i-(j+1)}t$. Then, $\beta_{ij}=\bar\kappa^{i-(j+1)}\frac tq$ and $\gamma_{ij}=\beta_{ij}-(p_--1)\le\bar\kappa^{-1}\frac tq-(p_--1)\le \kappa^{-1} t-(p_-^4-1)=-\gamma_0<0$.

Now, we iterate inequality \eqref{ineq-iterative} for $i=0,\cdots,j$ with $\rho=r_i$, $\sigma=r_{i+1}$, $r_i=\rho_1+2^{-i}(\rho_2-\rho_1)$. We get,
\[
\|u\|_{L^{\bar\kappa\bar\beta_{jj}}(B_{r_{j+1}})}\le \bar C \|u\|_{L^{\bar\beta_{0j}}(B_{r_0})}
\]
for a constant $\bar C$ depending on $j,q,\rho_1,\rho_2,M^{p_+^4-p_-^4},p_+^4,p_-^4$.
Thus, we get \eqref{t-bound} once we observe that $\rho_1\le r_{j+1}$, $r_0=\rho_2$, $\bar\kappa\bar\beta_{jj}=t$, $\bar\beta_{0j}=\bar\kappa^{-(j+1)}t$ and we choose $j$ large so that $\bar\kappa^{-(j+1)}t\le  t_0$.
\end{proof}

\bigskip

Now, by modifying the proof of Lemmas \ref{cacciopoli-lemma} and \ref{subhar} we will prove that  weak subsolutions  are locally bounded from above and weak supersolutions are locally bounded from bellow. This is already known when $p_1>N$ since weak super- and sub-solutions belong to $W^{1,p_1}(\Omega)\subset L^\infty(\Omega)$ if $p_1> N$.

We start with a variation of Cacciopoli inequality,

\begin{lemm}\label{cacciopoli-sin-acotacion} Let $u\in W^{1,p(x)}(B)$  such that $\Delta_{p(x)}
u\ge -H(x)(1+|u|)^{p(x)-1}$ in a ball $B$ and $\gamma\ge1$. Here $ H\ge0$ is a measurable function. Then, for $\eta
\in C_0^\infty(B)$ there holds that,
\begin{equation}\label{cacciopoli-F}\begin{aligned}
\int_{B}F_n(u_++1)|\nabla u_+|^{p_-}\eta^{p_+}&\le\int_{B} F_n(u_++1)\eta^{p_+}
+C\int_{B}u_+^{p(x)}F_n(u_++1)\eta^{p_+-p(x)}|\nabla\eta|^{p(x)}
\\&+C\int_{B} H(x)(u_++1)^{p(x)-1}G_n(u_++1)\eta^{p_+}\end{aligned}
\end{equation}
with $u_+=\max\{u,0\}$, $C=C(p_+,p_-)$. Here $p_+=\max_{\overline B}p$, $p_-=\min_{\overline B} p$.

In \eqref{cacciopoli-F}, the functions $F_n$ and $G_n$ are defined, for $s\ge1$, by
\[
G_n(s)=\int_1^s F_n(\tau)\,d\tau,
\]
\[
F_n(s)=\begin{cases}
s^{\gamma-1}\quad&\mbox{if}\quad 1\le s\le n,\\
n^{\gamma-1}\quad&\mbox{if}\quad  s\ge n.
\end{cases}
\]

\end{lemm}
\begin{proof} We proceed as in the proof of Lemma \ref{cacciopoli-lemma}. This time we take as test function \linebreak
 $\phi=G_n(u_++1)\eta^{p_+}\in W_0^{1,p(x)}(B)$ for every $\gamma\ge1$. We get,
\[
\begin{aligned}
\int F_n(u_++1)|\nabla u_+|^{p(x)}\eta^{p_+}&\le -p_+\int G_n(u_++1)\eta^{p_+-1}|\nabla u_+|^{p(x)-1}|\nabla\eta|\\
&\hskip2.5cm+\int H(x)(u_++1)^{p(x)-1}G_n(u_++1)\eta^{p_+}\\
&\le C\int u_+F_n(u_++1)\eta^{p_+-1}|\nabla u_+|^{p(x)-1}|\nabla\eta|\\
&\hskip2.5cm+\int H(x)(u_++1)^{p(x)-1}G_n(u_++1)\eta^{p_+}
\end{aligned}
\]
since $G_n(u_++1)=0$ if $u_+=0$ and $G_n(s)\le F_n(s)(s-1)$, as $F_n$ is a nondecreasing function in $[1,\infty)$.

Now, by applying Young inequality we get,
\[\begin{aligned}
\int F_n(u_++1)&|\nabla u_+|^{p(x)}\eta^{p_+}\\
&\le C\int u_+^{p(x)}F_n(u_++1)\eta^{p_+-p(x)}|\nabla\eta|^{p(x)}+\int H(x)(u_++1)^{p(x)-1}G_n(u_++1)\eta^{p_+}
\end{aligned}
\]
and the lemma is proved.\end{proof}

We can now prove the weak maximum principle. There holds,
\begin{lemm} \label{weak maximum principle}  Let $p$ log-H\"older continuous in $B_4$. Let $u\in W^{1,p(x)}(B_4)$  such that $\Delta_{p(x)}u\ge -H(x)(|u|+1)^{p(x)-1}$ in $B_4$ where $0\le H\in L^{q_0}(B_4)$ with $\max\{1,\frac N{p_-^{4}}\}<q_0\le \infty$.  Then, there exists $0<\bar \rho\le4$ such that, for every $0<\rho_1<\rho_2< \bar\rho$ and for every $0<t<\infty$ there holds that,
\begin{equation}\label{eq-subhar-elliptic-+}
\sup_{B_{\rho_1}}u_+\le C\Big(\frac{\rho_2}{\rho_2-\rho_1}\Big)^C
\Big(\fint_{B_{\rho_2}}(u_++1)^{t}\Big)^{1/t}.
\end{equation}

 The constant $C$  depends only on ${p_+^4.p_-^4,M^{p_+^4-p_-^4}}$, $ \| H(x)\|_{L^{q_0}(B_4)}$, $t$ and $q_0$. $\bar\rho$ depends on $q_0,p_-^4$ and the log-H\"older modulus of continuity of $p$ in $B_4$. Here $M=\big(\fint_{B_4} |u|^{p_-^4}\big)^{1/p_-^4}$.

\end{lemm}
\begin{proof} We start from \eqref{cacciopoli-F} with $\gamma\ge1$. Let now,
\[
L_n(s)=\int_1^s \big(F_n(\tau)\big)^{1/p_-}\,d\tau.
\]
Then,
\[
|\nabla L_n(u_++1)|^{p_-}= F_n(u_++1)|\nabla u_+|^{p_-}
\]
and we have,
\[
\begin{aligned}
\int|\nabla\big(\eta^{p_+/p_-}L_n(u_++1)\big)|^{p_-}&=\int F_n(u_++1)|\nabla u_+|^{p_-}\eta^{p_+}
+C\int L_n(u_++1)^{p_-}\eta^{p_+-p_-}|\nabla\eta|^{p_-}\\
&\le C\Big[\int F_n(u_++1)\eta^{p_+}
+\int u_+^{p}F_n(u_++1)\eta^{p_+-p}|\nabla\eta|^p\\
&+\int H(x)(u_++1)^{p-1}G_n(u_++1)\eta^{p_+}+\int L_n(u_++1)^{p_-}\eta^{p_+-p_-}|\nabla \eta|^{p_-}\Big]
\end{aligned}
\]

We bound, for $s\ge1$,
\[
\begin{aligned}
&F_n(s)\le s^{\gamma-1}\\
&L_n(s)\le F_n(s)^{1/p_-}(s-1)\quad\Rightarrow\quad L_n(u_++1)^{p_-}\le (u_++1)^{\gamma-1+p_-}\\
&s^{p-1}G_n(s)\le s^{p}F_n(s)\le s^{\gamma-1+p}\quad\Rightarrow\quad (u_++1)^{p-1}G_n(u_++1)\le (u_++1)^{\gamma-1+p}.
\end{aligned}
\]
Thus, by Sobolev inequality with $\kappa=\frac{\hat N}{\hat N-p_-}$ and $\hat N$ as in Lemma \ref{subhar},
\[
\begin{aligned}
\Big(\fint L_n(u_++1)^{\kappa p_-}\eta^{\kappa p_+}\Big)^{1/\kappa}&\le C\rho^{p_-}\fint|\nabla\big(\eta^{p_+/p_-}L_n(u_++1)\big)|^{p_-}\\
&\le C\rho^{p_-}\Big[\fint (u_++1)^{\gamma-1}\eta^{p_+}+\fint (u_++1)^{\gamma-1+p_-}\eta^{p_+-p_-}|\nabla \eta|^{p_-}\\
&+\fint (u_++1)^{\gamma-1+p}\eta^{p_+-p}|\nabla\eta|^{p}+\fint H(x)(u_++1)^{\gamma-1+p}\eta^{p_+}
\Big].
\end{aligned}
\]

We take $\bar\rho\le4$ such that $p_+^{\bar\rho}-p_-^{\bar\rho}<\min\{p_-^4/q', p_-^4/r_0\}$ with $q'$ and $r_0$ as in the proof of Lemma \ref{subhar}. Let $0< \sigma<\rho\le\bar\rho$,  $\eta\in C_0^\infty(B_{\rho})$, $0\le \eta\le 1$, $\eta\equiv1$ in $B_{\sigma}$, $|\nabla\eta|\le \frac C{\rho-\sigma}$ and let us proceed as in the proof of Lemma \ref{subhar}.

Observe that, by the choice of $\bar\rho$, there exists $s\ge p_+-p_-$ such that $sq'\le p_-^4$ and $sr_0\le p_-^4$ and we fix such an $s$ for the next steps.

We can proceed with the proof as long as  $u_+\in L^{q(\gamma-1+p_-)}(B_\rho)$ with $q$ as in the proof of Lemma \ref{subhar}. This is the case for any value of $\gamma\ge 1$ if $p_-\ge N$. If instead, $p_-<N$ there holds that $\hat N=N$ and $1<q<\frac N{N-p_-}$. Therefore, if we take $\gamma=1$ we will have $u_+\in L^{q(\gamma-1+p_-)}(B_\rho)$ as needed in order to continue with the estimates. Thus, we get, with $\beta=\gamma-1+p_-$,
\[
\begin{aligned}
\Big(\fint_{B_\sigma} L_n(u_++1)^{\kappa p_-}\Big)^{1/\kappa\beta}\le C\Big(\frac\rho\sigma\Big)^{N/\kappa\beta}\Big(\frac\rho{\rho-\sigma}\Big)^{p_+/\beta}
\Big(\fint_{B_\rho}(u_++1)^{q\beta}\Big)^{1/q\beta}.
\end{aligned}
\]

Since the right hand side is independent of $n$ and finite as long as $u_+\in L^{q\beta}(B_\rho)$ (for instance, if $\beta=p_-$ so that $q\beta\le p_-^*$), we can pass to the limit and get
\[
\Big(\fint_{B_\sigma}(u_++1)^{\kappa\beta}\Big)^{1/\kappa\beta}\le C\Big[1+(1+\beta)^{p_-/\beta}\Big(\frac\rho\sigma\Big)^{N/\kappa\beta}\Big(\frac\rho{\rho-\sigma}\Big)^{p_+/\beta}
\Big(\fint_{B_\rho}(u_++1)^{q\beta}\Big)^{1/q\beta}\Big].
\]

In fact, there holds that $$L_n(s)\to \frac{p_-}{\gamma-1+p_-}\big(s^{\frac{\gamma-1+p_-}{p_-}}-1\big)=\frac{p_-}\beta\big(s^{\frac\beta{p_-}}-1\big).$$

As in Lemma \ref{subhar} we call $\bar\kappa=\frac\kappa{q}$, $\bar\beta=q\beta$ and get
\begin{equation}\label{iterative-general}
\begin{aligned}
\Big(\fint_{B_\sigma}(u_++1)^{\bar\kappa\bar\beta}\Big)^{1/\bar\kappa\bar\beta}&\le C\Big[1+(1+\beta)^{qp_-/\bar\beta}\Big(\frac\rho\sigma\Big)^{N/\bar\kappa\bar\beta}
\Big(\frac\rho{\rho-\sigma}\Big)^{qp_+/\bar\beta}
\Big(\fint_{B_\rho}(u_++1)^{\bar\beta}\Big)^{1/\bar\beta}\Big]\\
&\le 2C(1+\beta)^{qp_-/\bar\beta}\Big(\frac\rho\sigma\Big)^{N/\bar\kappa\bar\beta}
\Big(\frac\rho{\rho-\sigma}\Big)^{qp_+/\bar\beta}
\Big(\fint_{B_\rho}(u_++1)^{\bar\beta}\Big)^{1/\bar\beta}.
\end{aligned}
\end{equation}

Now we can proceed as in Lemma \ref{subhar} with the iterative process. In each step we use that $u_+\in L^{\bar\beta_j}(B_{r_j})$ in order to deduce that $u_+\in L^{\bar\beta_{j+1}}(B_{r_{j+1}})$ and continue with the iteration, starting with $\bar\beta_0={p_-^4}^*$.

In this way we prove \eqref{eq-subhar-elliptic-+} for $t={p_-^4}^*$ if $p_-^4<N$, any positive number if $p_-^4\ge N$. Now, if $p_-^4<N$ and $0<t<{p_-^4}^*$ we use Lemma \ref{HKM} to get the result. In particular, for $\rho_2=\bar\rho$ we get \eqref{eq-subhar-elliptic-+} with $t=p_-^{4}$. So that, $u\in L^\infty(B_{\widetilde\rho})$ for any $\widetilde\rho<\bar\rho$. Therefore,
$u_+\in L^t(B_{\rho_2})$ for every $t>0$ if $\rho_2<\bar\rho$ and we can proceed with the proof without any restriction on $t$.  So that, \eqref{eq-subhar-elliptic-+} holds for every $t>0$ if $0<\rho_1<\rho_2<\bar\rho$.
\end{proof}

In a similar way, we can prove
\begin{lemm}\label{cacciopoli-sin-acotacion-u-} Let $u\in W^{1,p(x)}(B)$  such that $\Delta_{p(x)}
u\le H(x)(|u|+1)^{p(x)-1}$ in a ball $B$ and $\gamma\ge1$. Here $ H\ge0$ is a measurable function. Then, for $\eta
\in C_0^\infty(B)$ there holds that,
\begin{equation}\label{cacciopoli-F-u-}\begin{aligned}
\int_{B}F_n(u_-+1)|\nabla u_-|^{p_-}\eta^{p_+}&\le\int_{B} F_n(u_-+1)\eta^{p_+}
+C\int_{B}u_-^{p(x)}F_n(u_-+1)\eta^{p_+-p(x)}|\nabla\eta|^{p(x)}
\\&+C\int_{B} H(x)u_-^{p(x)-1}G_n(u_-+1)\eta^{p_+}\end{aligned}
\end{equation}
with $u_-=\max\{-u,0\}$, $C=C(p_+,p_-)$. Here $p_+=\max_{\overline B}p$, $p_-=\min_{\overline B} p$.

In \eqref{cacciopoli-F-u-}, the functions $F_n$ and $G_n$ are defined as in Lemma \ref{cacciopoli-sin-acotacion}.
\end{lemm}

We also have,
\begin{lemm} \label{weak maximum principle u-}  Let $p$ log-H\"older continuous in $B_4$. Let $u\in W^{1,p(x)}(B_4)$  such that $\Delta_{p(x)}u\le H(x)(|u|+1)^{p(x)-1}$ in $B_4$ where $0\le H\in L^{q_0}(B_4)$ with  $ \max\{1,\frac N{p_-^{4}}\}<q_0\le \infty$.  Then, there exists $\bar\rho$ such that for every $0<\rho_1<\rho_2< \bar\rho<4$ and any $0<t<\infty$ there holds that,
\begin{equation}\label{eq-subhar-elliptic--}
\sup_{B_{\rho_1}}u_-\le C\Big(\frac{\rho_2}{\rho_2-\rho_1}\Big)^C
\Big(\fint_{B_{\rho_2}}(u_-+1)^{t}\Big)^{1/t}.
\end{equation}

 The constant $C$  depends  on $t,{p_+^4.p_-^4,M^{p_+^4-p_-^4}}$, $ \| H(x)\|_{L^{q_0}(B_4)}$ and $q_0$. $\bar\rho$ depends on $q, r_0,p_-^4$ for certain $q,r_0\in(1,\infty)$ such that $\frac1{q_0}+\frac1q+\frac1{r_0}=1$; and the log-H\"older modulus of continuity of $p$ in $B_4$. Here $M=\big(\fint_{B_4} |u|^{p_-^4}\big)^{1/p_-^4}$.

\end{lemm}

We conclude
\begin{prop}[Weak maximum principle]\label{local bounds} Let $\Omega\subset\R^N$ bounded and $p$ log-H\"older continuous in $\Omega$.  Let $u\in W^{1,p(x)}(\Omega)$  such that $\Delta_{p(x)}u\ge- H(x)(|u|+1)^{p(x)-1}$ in $\Omega$ with $0\le H\in L^{q_0(x)}(\Omega)$ with $q_0\in C(\Omega)$, $\max\{1,\frac N{p(x)}\}<q_0(x)$ for every $x\in\Omega$. Let $\Omega'\subset\subset\Omega$. Then, $u$ is bounded from above in $\Omega'$.
More precisely, for every $0<t<\infty$,
\begin{equation}\label{uniform bound}
\sup_{\Omega'}u\le \widetilde C \Big[1+\| u\|_{L^{t}({\Omega'')}}\Big]
\end{equation}
where $\Omega''=\big\{x\in\Omega\,,\,\mbox{dist}(x,\Omega')<\frac12\mbox{dist}
(\Omega',\partial\Omega)\big\}$.
 Here $\widetilde C$ depends on $t$, $\Omega'$, $p(x)$, $q_0(x)$, $\||H|^{q_0(x)}\|_{L^1(\Omega)}$ and $\||u|^{p(x)}\|_{L^1(\Omega)}$.

\smallskip

If $\Delta_{p(x)}u\le H(x)(|u|+1)^{p(x)-1}$ in $\Omega$, there holds that $u$ is bounded from below by $-\widetilde C \Big[1+\| u\|_{L^{t}(B_{\Omega'')}}\Big]$.

\end{prop}
\begin{proof} Let $0<R= \min\{1,\frac14\mbox{dist}(\Omega',\partial\Omega)\}$.  For $x_0\in\Omega'$, let $\bar u(x)=\frac{u(x_0+Rx)}R$, $\bar p(x)=p(x_0+Rx)$ and $\bar H(x)=RH(x_0+Rx)$. Then, $\Delta_{\bar p(x)}\bar u\ge- \bar H(x)(|\bar u|+1)^{\bar p(x)-1}$ in $B_4$.

We claim that there exists $0<\bar r<1$,  $\bar q_0>0$,  possibly depending on $x_0$, such that $q_0(x_0+Rx)\ge \bar q_0>\max\{1,\frac N{\bar p_-^{4\bar r}}\}$ for every $x\in B_{4\bar r}$.
In fact, if $\bar p(0)<N$ we let $\rho_1$ such that $\bar p(x)<N$ in $B_{ 4\rho_1}$. Then, let $\ep>0$ such that $q_0(x_0)\ge \frac N{\bar p(0)}+3\ep$ and $\rho_2\le \rho_1$  such that $q_0(x_0+Rx)\ge \bar q_0:=\frac N{\bar p(0)}+2\ep$ in $B_{4\rho_2}$. Finally, $\bar r\le\rho_2$ such that $\frac N {\bar p(x)}-\frac N{\bar p(0)}<\ep$ in $B_{4\bar r}$. So, in $B_{4\bar r}$ we have $q_0(x_0+Rx)\ge \bar q_0>\max\{1,\frac N{p_-^{4\bar r}}\}$.

 Now, if $\bar p(0)\ge N$, we let first $\rho_1$ and $\ep>0$ such that $q_0(x_0+Rx)\ge \bar q_0:=1+2\ep$ in $B_{4\rho_1}$ and then, $\bar r\le\rho_1$ such that $\frac N{\bar p(x)}\le 1+\ep$ in $B_{4\bar r}$. So we have $q_0(x_0+Rx)\ge \bar q_0>\max\{1,\frac N{\bar p_-^{4\bar r}}\}$ in $B_{4\bar r}$.

We can assume that $\bar r$ is small so that,
$\bar p_+^{4\bar r}- \bar p_-^{4\bar r}<\min\{p_1/q',p_1/r_0\}$ with $q$ and $r_0$ as in Lemma~\ref{subhar}, ($\frac 1{\bar q_0}+\frac1q+\frac 1{r_0}=1$). Then, by Lemma \ref{weak maximum principle} (observe that we may take $\bar\rho=4\bar r$ in that lemma by the conditions imposed to $\bar r$), for every $0<t<\infty$,
\[
\sup_{B_{\bar r}}\bar u\le C\Big[1+\|\bar u\|_{L^{t}(B_{2\bar r)}}\Big]
\]
with $C$ depending on $t$, $\bar r$, $p_1$, $p_2$, the log-H\"older modulus of continuity of $p$ in $\Omega''$, $\bar q_0$, $r_0$, $\|\bar H\|_{L^{\bar q_0}(B_{4\bar r})}$ and $M^{p_2-p_1}$ where $M=\|u\|_{L^{p_1}(\Omega'')}$.

Observe that $\|\bar H\|_{L^{\bar q_0}(B_{4\bar r})}\le C\big[1+\|||H|^{q_0(x)}\|_{L^1(\Omega)}\big]^{1/\inf_{\Omega}q_0}$ with $C$ depending on $R$, $\bar r$ and $q_0$.

Thus, any point $x_0\in\Omega'$ has a neighborhood $B_{\bar r R}(x_0)$ where
\[
\sup_{B_{R \bar r}(x_0)} u\le \widetilde C\Big[1+\| u\|_{L^{t}(B_{2R\bar r)(x_0)}}\Big]
\]
with $\widetilde C$ depending on  the neighborhood, on $t$, $p(x)$, $q(x)$, $ \|  |H(x)|^{q_0(x)}\|_{L^1(\Omega)}^{1/\inf_{\Omega}q_0}$ and $\||u|^{p(x)}\|^{1/\inf_{\Omega}p}_{L^1(\Omega)}$.

Since $\Omega'$ is compact, we get the result on the upper bound.

Analogously, if $\Delta_{p(x)}u\le H(x)|u|^{p(x)-1}$ in $\Omega$ we find a similar uniform bound from above for $u_-$ in $\Omega'$. So, we get the lower bound.
\end{proof}

\medskip

As a corollary we get local bounds for  weak solutions to \eqref{p(x)-f}. There holds,
\begin{coro} Let $\Omega\subset\R^N$ be bounded and $p$ log-H\"older continuous in $\Omega$. Let $ u\in W^{1,p(x)}(\Omega)$ a weak solution to
\[
\Delta_{p(x)}u=f\quad\mbox{in}\quad\Omega
\]
with $f\in L^{q_0(x)}(\Omega)$ with $q_0\in C(\Omega)$ such that $\max\{1,\frac N{p(x)}\}<q_0(x)$ in $\Omega$. Then, $u$ is locally bounded.
\end{coro}
\begin{proof} Let $ H(x)=|f(x)|$. Then,
\[
|\Delta_{p(x)} u|=|f(x)|
\le H(x) (|u|+1)^{p(x)-1}.
\]

The result follows by applying Propositon \ref{local bounds}.
\end{proof}

\bigskip

Now, we prove Harnack inequality for solutions to \eqref{p(x)-f}.

\begin{proof}[Proof of Theorem \ref{har}] Without loss of generality we may assume that $x_0=0$.

Let $u$ and $f$ as in the statement. Let $\bar p(x)=p(Rx)$.

If $f\not\equiv0$ in $B_{4R}$, let $\widetilde H(x)=R|f(Rx)|$,
$$\bar u(x)=1+\|\widetilde H\|_{L^{q_0}(B_4)}^{\frac1{p_-^{4R}-1}}+\frac{u(Rx)}R$$ and $$H(x)=\frac{\widetilde H(x)}{\|\widetilde H\|_{L^{q_0}(B_4)}}.$$

If $f\equiv0$ in $B_{4R}$, let
$$\bar u(x)=1+\frac{u(Rx)}R$$ and $$H(x)\equiv0.$$
Then,
\[
\max_{B_4}\bar p=\max_{B_{4R}} p,\quad \min_{B_4}\bar p=\min_{B_{4R}} p,\]
and for $x,y\in B_4$,
\[|\bar p(x)-\bar p(y)|\le \omega_{4R}(R|x-y|)\le \omega_{4R}(|x-y|)
\]
if $0<R\le 1$
and,
\[\big|\Delta_{\bar p(x)}\bar u(x)\big|=\big|Rf( Rx)\big|\le  H(x)\Big(1+\|\widetilde H\|_{L^{q_0}(B_4)}^{\frac1{p_-^{4R}-1}}+\big(\frac{u(Rx)}R\big)\Big)^{ p_-^{4R}-1}\le H(x) \bar u^{\bar p(x)-1}.
\]

Therefore, we can apply Lemmas \ref{subhar} and \ref{superhar} (recall that we already know that $u$ is locally bounded and therefore, $\bar u$ is bounded in $B_4$) with $\rho_1=1$, $\rho_2=2$ and $t=t_0$
to obtain
\[
\sup_{B_1}\bar u\le C\Big(\fint_{B_{2}}\bar u^{t_0}\Big)^{1/t_0}\le
C\inf_{B_{1}}\bar u.
\]

Recall that $\| H\|_{L^{q_0}(B_4)}=1$ or $\| H\|_{L^{q_0}(B_4)}=0$. Thus, $C$ is independent of $H$ and so it depends on $f$  only through its dependence on $\bar u$.

\smallskip

Since $\bar u(x)=\frac{u(Rx)+R+R\|\widetilde H\|_{L^{q_0}(B_4)}^{\frac1{p_-^{4R}-1}}}R$ there holds that,
\[
\sup_{B_{R}}u\le C\big[\inf_{B_{R}} u+R+R\|\widetilde H\|_{L^{q_0}(B_4)}^{\frac1{p_-^{4R}-1}}\big].
\]

Now, $\|\widetilde H\|_{L^{q_0}(B_4)}=R^{1-\frac N{q_0}}\|f\|_{L^{q_0}(B_{4R})}$. And,
\begin{equation}\label{final}
\begin{aligned}
\bar M_1^{\bar p_+^{4}-\bar p_-^4}: =\Big(\fint_{B_4}\bar u^{sq'}\Big)^{\frac{\bar p_+^{4}-\bar p_-^4}{sq'}}
&\le C\Big[R^{-1}\Big(\fint_{B_{4R}}u^{sq'}\Big)^{1/sq'}+1+\|\widetilde H\|_{L^{q_0}(B_4)}^{\frac1{p_-^{4R}-1}}\Big]^{p_+^{4R}-p_-^{4R}}\\
&\le C\Big[\Big(\|u\|_{L^{sq'}(B_{4R})}+1+\big(R^{1-\frac N{q_0}}\|f\|_{L^{q_0}(B_{4R})}\big)^{\frac1{p_-^{4R}-1}}\Big]^{p_+^{4R}-p_-^{4R}}
\end{aligned}
\end{equation}
since $R^{-(p_+^{4R}-p_-^{4R})}\le C$  with $C$ independent of $R$. In particular, $\bar M_1^{\bar p_+^{4}-\bar p_-^4}$ is bounded independently of $R$.

The same kind of bound holds for $\bar M_2^{\bar p_+^4-\bar p_-^4}$.

So, the theorem is proved.
\end{proof}

%
%
%
%
%
\begin{rema} Observe that, since $q_0>\frac N{p_-^{4R}}$ there holds that
\[
1+\frac{1-\frac N{q_0}}{p_-^{4R}-1}>1-\frac{p_-^{4R}-1}{p_-^{4R}-1}=0.
\]

Thus, \eqref{har-ineq} can be stated as:
\begin{equation}\label{har-ineq-alt}
\sup_{B_{R}(x_0)}u\le C\big[\inf_{B_{R}(x_0)} u+R+R^\delta L\big]
\end{equation}
for a certain $\delta>0$.

The power $\delta$ can be made independent of $R$. In fact, we may take $\delta=
1+\frac{1-\frac N{q_0}}{p_1-1}>0$ if $N\ge q_0>\frac N{p_1}$ with $p_1=\inf_\Omega p$ and $\delta=
1+\frac{1-\frac N{q_0}}{p_2-1}>1$ if $ q_0> N$ with $p_2=\sup_\Omega p$.

Here $L:= \big(1+\|f\|_{L^{q_0}(\Omega)}\big)^{\frac1{p_1-1}}\ge \|f\|_{L^{q_0}(B_{4R})}^{\frac1{p_-^{4R}-1}}$.
\end{rema}

\medskip

\begin{rema} Observe that, since $p$ is continuous in $\Omega$, if $R$ is small enough, we may choose $s\ge p_+^{4R}-p_-^{4R}$ such that $sq'\le p_-^{4R}$ and $sr_0\le p_-^{4R}$. So, the constant $C$ in \eqref{eq-subhar} depends on $u$ only through $\||u|^{p(x)}\|_{L^{1}(B_{4R}(x_0))}^{p_+^{4R}-p_-^{4R}}$.

A similar comment applies to \eqref{har-ineq} and \eqref{har-ineq-alt}.

\end{rema}

From Harnack inequality we get H\"older continuity of  weak solutions. There holds,
\begin{coro}\label{coro-holder}  Let $\Omega\subset\R^N$ bounded and $p$ log-H\"older continuous in $\Omega$ with $1<p_1\le p(x)\le p_2<\infty$ in $\Omega$. Let $f\in L^{q_0}(\Omega)$ with $\max\{1,\frac N{p_1}\}<q_0\le \infty$. Let $u$ be a  weak solution to
\begin{equation}\label{eqn}
\Delta_{p(x)} u=f\quad\mbox{in}\quad\Omega.
\end{equation}
Then, $u$ is locally H\"older continuous in $\Omega$ with constant and exponent depending only on the compact subdomain and on $p(x),\ q_0$, $\|f\|_{L^{q_0}(\Omega)}$ and $M^{{p_2}-{p_1}}$ where $M=\||u|^{p(x)}\|_{L^1(\Omega)}$.
\end{coro}
\begin{proof} Once we have Harnack inequality, the proof is standard. Let $\Omega'\subset\subset\Omega$. There exist $L,R_0,\delta>0$ such that for any nonnegative weak solution $v$ of \eqref{eqn}, any $x_0\in\Omega'$ and $0<R\le R_0$,
\begin{equation}\label{harnack}
\sup_{B_{R}(x_0)} v\le C\big[\inf_{B_{R}(x_0)} v+ R+R^\delta L\big].
\end{equation}

Now, apply \eqref{harnack} with $R=2^{-(j+1)}R_0$ to the functions $v_1=M_j-u(x)$ and $v_2=u(x)-m_j$ where $M_j=\sup_{B_{2^{-j}R_0}(x_0)}u$, $m_j=\inf_{B_{2^{-j}R_0}(x_0)}u$ to obtain that
\[
\mbox{osc}_{j+1}\ u\le\nu\ \mbox{osc}_j \ u+C(L) R^\delta
\]
with $0<\nu<1$ to obtain the result (see \cite{GT} for the details). The constant and exponent of the H\"older continuity in $\Omega'$ depend only on $\nu$, $C(L)$ and $\delta$.
\end{proof}

\medskip

By applying Corollary \ref{coro-holder} on small enough neighborhoods of points  $x_0\in\Omega'\subset\subset\Omega$ --as in Proposition \ref{local bounds}-- we get local H\"older continuity with variable $q_0$. There holds,
\begin{coro}\label{holder-q(x)}Let $\Omega\subset\R^N$ bounded and $p$ log-H\"older continuous in $\Omega$ with $1<p_1\le p(x)\le p_2<\infty$ in $\Omega$. Let $f\in L^{q_0(x)}(\Omega)$ with $q_0\in C(\Omega)$ and $\max\{1,\frac N{p(x)}\}<q_0(x)$ in $\Omega$. Let $u$ be a  weak solution to
\begin{equation}\label{eqn}
\Delta_{p(x)} u=f\quad\mbox{in}\quad\Omega.
\end{equation}
Then, $u$ is locally H\"older continuous in $\Omega$ with constant and exponent depending only on  the compact subdomain and on $p(x),\ q_0(x)$, $\||f|^{q_0(x)}\|_{L^{1}(\Omega)}$ and $\||u|^{p(x)}\|_{L^1(\Omega)}^{p_2-p_1}$.
\end{coro}
%
%
%

\end{section}

\bigskip

\begin{section}{Harnack inequality for solutions to general elliptic equations}
\label{sect-general-eqs}
In this section we will generalize the results of Section \ref{sect-p-laplace} to elliptic equations with $p(x)-$type growth. More precisely,
\begin{equation}\label{elliptic}
\mbox{div\,}A(x ,u, \nabla u)=B(x,u,\nabla u)\quad\mbox{in}\quad\Omega.
\end{equation}

We assume that for every  $M_0>0$ there exist a constant $\alpha$ and nonnegative functions $g_0, C_0\in L^{q_0}(\Omega)$,  $g_1,C_1\in L^{q_1}(\Omega)$, $f,C_2\in L^{q_2}(\Omega)$, $K_2^{p(x)}\in L^{t_2}(\Omega)$, $K_1\in L^\infty(\Omega)$ for some $\max\{1,\frac N{p_1-1}\}<q_0,q_1\le\infty$ ($p_1=\inf_\Omega p$),
$\max\{1,\frac N{p_1}\}<q_2,t_2\le\infty$,     such that, for every $x\in\Omega$, $|s|\le M_0$, $\xi\in\R^N$,
\begin{enumerate}
\item $A(x,s,\xi)\cdot\xi\ge \alpha|\xi|^{p(x)}-C_0|s|^{p(x)}-g_0(x)$.

\item $\big|A(x,s,\xi)\big|\le g_1(x)+C_1|s|^{p(x)-1}+K_1|\xi|^{p(x)-1}$.

\item $\big|B(x,s,\xi)\big|\le f(x)+C_2|s|^{p(x)-1}+K_2|\xi|^{p(x)-1}$

\end{enumerate}

\medskip

We start with a Cacciopoli type estimate.
\begin{lemm}\label{lemma-cacciopoli-elliptic} Let $u\ge1$ and bounded be such that $\mathrm{div}A(x ,u, \nabla u)\ge - \big(H_2(x)u^{p(x)-1}+G_2(x)|\nabla u|^{p(x)-1}\big)$ in a ball $B$ and  $\gamma>0$  or $\mathrm{div}A(x ,u, \nabla u)\le  H_2(x)u^{p(x)-1}+G_2(x)|\nabla u|^{p(x)-1}$ in a ball $B$ and $\gamma<0$. Assume that there exists a positive constant $ \alpha$  such that,
\begin{enumerate}
\item $A(x,u(x),\nabla u(x))\cdot\nabla u(x)\ge \alpha|\nabla u(x)|^{p(x)}-H_0(x) u(x)^{p(x)}$ in $B$.

	\item $\big|A(x,u(x),\nabla u(x))\big|\le H_1(x)u^{p(x)-1}+G_1(x)|\nabla u|^{p(x)-1}$ in $B$
	\end{enumerate}
for certain nonnegative measurable functions $H_i$, $G_j$, $i=0,1,2$, $j=1,2$.

Let $\eta\in C_0^\infty(B)$, $\eta\ge0$. Then, there exists a constant $C$ that depends only on $p_+=\sup_B p$, $p_-=\inf_B p$ and $\alpha$ such that,
\begin{equation}\label{cacciopoli-elliptic}
\begin{aligned}\int u^{\gamma-1}\eta^{p_+}&|\nabla u|^{p_-}\le \int u^{\gamma-1}\eta^{p_+}+C\left[|\gamma|^{-1}\int (H_0+H_2)u^{\gamma+p(x)-1}\eta^{p_+}\right.\\
&\left.+|\gamma|^{-1}\int H_1u^{\gamma+p(x)-1}\eta^{p_+-1}|\nabla \eta|+|\gamma|^{-p_+} \int G_1^{p(x)}u^{\gamma+p(x)-1}\eta^{p_+-p(x)}|\nabla \eta|^{p(x)}\right.\\
&\left.+|\gamma|^{-p_+} \int G_2^{p(x)}u^{\gamma+p(x)-1}\eta^{p_+-p(x)}\right].
 \end{aligned}
\end{equation}	
Here $p_+=p_+^B, p_-=p_-^B$.
\end{lemm}

\begin{proof} Let us consider the case of $\gamma>0$. As in the proof of Lemma \ref{cacciopoli} we take $u^\gamma\eta^{p_+}$ as test function. Then,
\[\begin{aligned}
\alpha\gamma\int u^{\gamma-1}\eta^{p_+}|\nabla u|^{p(x)}&\le
-p_+\int H_1u^{\gamma+p(x)-1}\eta^{p_+-1}|\nabla \eta|-p_+\int G_1u^{\gamma}\eta^{p_+-1}|\nabla u|^{p(x)-1}|\nabla \eta|\\
&+\int H_2u^{\gamma+p(x)-1}\eta^{p_+}+\int G_2u^{\gamma}\eta^{p_+}|\nabla u|^{p(x)-1}
+\int H_0 u^{\gamma+p-1}\eta^{p_+}
\end{aligned}\]

As in the proof of Lemma \ref{cacciopoli},
\[\begin{aligned}
\int G_1u^{\gamma}\eta^{p_+-1}|\nabla u|^{p(x)-1}|\nabla \eta|&\le
\frac{\alpha\gamma}{4p_+}\int u^{\gamma-1}\eta^{p_+}|\nabla u|^{p(x)}\\&+C\gamma^{-p_++1}\int G_1^{p(x)}u^{\gamma+p(x)-1}\eta^{p_+-p(x)}|\nabla \eta|^{p(x)}
\end{aligned}\]

Similarly,
\[\begin{aligned}
\int G_2u^{\gamma}\eta^{p_+}|\nabla u|^{p(x)-1}&\le
\frac{\alpha\gamma}4\int u^{\gamma-1}\eta^{p_+}|\nabla u|^{p(x)}+C\gamma^{-p_++1}\int G_2^{p(x)}u^{\gamma+p(x)-1}\eta^{p_+-p(x)}.
\end{aligned}\]

Hence, since
\[\int u^{\gamma-1}\eta^{p_+}|\nabla u|^{p_-}\le \int u^{\gamma-1}\eta^{p_+}+\int u^{\gamma-1}\eta^{p_+}|\nabla u|^{p(x)}
\]
we have \eqref{cacciopoli-elliptic}.

The case of $\gamma<0$ is done in a similar way.
\end{proof}

	\medskip

Once we have a Cacciopoli type estimate we can get results  similar to Lemmas \ref{subhar} and \ref{superhar}.
%


So, we have
\begin{lemm}\label{subhar-elliptic} Let $p$ log-H\"older continuous in $B_4$. Let $u\ge 1$ bounded  be such that $\mathrm{div}A(x ,u, \nabla u)\ge - \big(H_2(x)u^{p(x)-1}+G_2(x)|\nabla u|^{p(x)-1}\big)$ in $B_4$. Assume that there exists a positive constant $ \alpha$  such that,
\begin{enumerate}
\item $A(x,u(x),\nabla u(x))\cdot\nabla u(x)\ge \alpha|\nabla u(x)|^{p(x)}-H_0(x) u(x)^{p(x)}$ in $B_4$.

	\item $\big|A(x,u(x),\nabla u(x))\big|\le H_1(x)u^{p(x)-1}+G_1(x)|\nabla u|^{p(x)-1}$ in $B_4$.
	\end{enumerate}
	Here  $H_i\in L^{q_i}(B_4)$ $i=0,1,2$,  $G_2^{p(x)}\in L^{t_2}(B_4)$ with $\max\{1,\frac N{p_-^{4R}}\}<q_i,t_2\le\infty$ for $i=0,2$,  $\max\{1,\frac N{p_-^{4R}-1}\}<q_1\le\infty$, $G_1\in L^\infty(B_4)$ and they are nonnegative.
Then, for every $0<\sigma<\rho\le4$  and $t>0$ there holds that,
\begin{equation}\label{eq-subhar-elliptic-alt}
\sup_{B_{\rho_1}}u\le C\Big(\frac{\rho_2}{\rho_2-\rho_1}\Big)^C
\Big(\fint_{B_{\rho_2}}u^t\Big)^{1/t}
\end{equation}

The constant $C$ depends only on $s,p_+^4,p_-^4,q_i,t_2,t,\alpha,\|H_i\|_{L^{q_i}(B_4)}$, $\|G_1^{p(x)}\|_{L^{\infty}(B_4)}$,
 $\|G_2^{p(x)}\|_{L^{t_2}(B_4)}$, $\|u\|_{L^{sq'}(B_4)}^{p_+^4-p_-^4}$, $\|u\|_{L^{ss_2}(B_4)}^{p_+^4-p_-^4}$ and $\|u\|_{L^{sr_i}(B_4)}^{p_+^4-p_-^4}$ for certain $q'=\frac q{q-1}$, $r_0\in(1,\infty)$ with
 $\frac1{q_i}+\frac1q+\frac1{r_i}=1$ $i=0,1,2$, $\frac1{t_2}+\frac1q+\frac1{s_2}=1$. Here $s\ge p_+^4-p_-^4$ is arbitrary.
\end{lemm}
\begin{proof}
We proceed as in the proof of Lemma \ref{subhar}. If $p_-^4\ge N$ we choose $\hat N=N$. If $p_-^4<N$ we choose $\hat N$ such that $p_-^4<\hat N<q_ip_-^4$ for $i=0,1,2$ and also $p_-^4<\hat N<t_2p_-^4$. Then, we choose $1<q<\frac{\hat N}{\hat N-p_-^4}$ such that $\frac1{q_i}+\frac1 q<1$ for $i=0,1,2$ and $\frac1{t_2}+\frac1 q<1$. Finally, we take $r_i\in (1,\infty)$ such that $\frac1{q_i}+\frac 1q+\frac1{r_i}=1$ and $s_2\in(1,\infty)$ such that $\frac1{t_2}+\frac 1q+\frac1{s_2}=1$.

We will be calling $M_{i+2}=\Big(\fint_{B_4} u^{sr_i}\Big)^{1/sr_i}$, $i=0,1,2$,
$M_1=\Big(\fint_{B_4} u^{sq'}\Big)^{1/sq'}$, $M_5=\Big(\fint_{B_4} u^{ss_2}\Big)^{1/ss_2}$, $M=\sum_{j=1}^5 M_j$.

The terms involving $H_0, H_2$ are treated exactly as the term with $H$ in Lemma \ref{subhar}. The term involving $H_1$ is treated similarly. We have,
\[\begin{aligned}
\fint H_1(x) u^{\gamma+p(x)-1}\eta^{p_+-1}|\nabla\eta|&\le \frac C{\rho-\sigma}\Big(\fint_{B\rho}H_1^{q_1}\Big)^{1/q_1}\Big(\fint_{B_\rho}u^{q(\gamma+p_--1)}\Big)^{1/q}\Big(\fint u^{r_1(p^+-p_-)}\Big)^{1/r_1}\\
&\le \frac C{(\rho-\sigma)^{1+\frac N{q_1}}}\|H_1\|_{B_4}M_3^{p_+-p_-}\Big(\fint_{B_\rho}u^{q(\gamma+p_--1)}\Big)^{1/q}\\
&\le \frac C{(\rho-\sigma)^{p_+}}\|H_1\|_{B_4}M_3^{p_+-p_-}\Big(\fint_{B_\rho}u^{q(\gamma+p_--1)}\Big)^{1/q}
\end{aligned}\]
since $1+\frac N{q_1}<p_-^4\le p_-\le p_+$.

And,
\[
\begin{aligned}
\fint G_2^{p(x)}u^{\gamma+p(x)-1}\eta^{p_+}&\le \rho^{-\frac N{t_2}}\|G_2^{p(x)}\|_{L^{t_2}(B_4)}\Big(\fint_{B_\rho}u^{q(\gamma+p_--1)}\Big)^{1/q}
\Big(\fint_{B_\rho}u^{s_2(p_+-p_-)}\Big)^{1/s_2}\\
&\le \frac C{(\rho-\sigma)^{-p_+}}M_5^{p_+-p_-}\|G_2^{p(x)}\|_{L^{t_2}(B_4)}
\end{aligned}
\]
since $\frac N{t_2}<p_-^4\le p_-\le p_+$,, $0<\rho-\sigma<\rho<4$.

Let us now look at the term involving $G_1$ which is bounded by
\[
\frac C{(\rho-\sigma)^{p_+}}\| G_1^{p(x)}\|_{L^\infty(B_4)}M_1^{p_+-p_-}\Big(
\fint u^{q(\gamma+p_--1)}\Big)^{1/q}.
\]

Now, the proof follows with no change.
\end{proof}

Also,
\begin{lemm}[Weak Harnack] \label{superhar-elliptic} Let $p$ log-H\"older continuous in $B_4$. There exist $t_0>0$  such that, for $s\ge p_+^4-p_-^4$ there exists $C$ such that, if $u\ge 1$  and bounded is such that $\mathrm{div}A(x ,u, \nabla u)\le  H_2(x)u^{p(x)-1}+G_2(x)|\nabla u|^{p(x)-1}$ in $B_4$ and there exists a positive constant $ \alpha$  such that,
\begin{enumerate}
\item $A(x,u(x),\nabla u(x))\cdot\nabla u(x)\ge \alpha|\nabla u(x)|^{p(x)}-H_0(x) u(x)^{p(x)}$ in $B_4$

	\item $\big|A(x,u(x),\nabla u(x))\big|\le H_1(x)u^{p(x)-1}+G_1(x)|\nabla u|^{p(x)-1}$ in $B_4$
	\end{enumerate}
	with  $H_i\in L^{q_i}(B_4)$, $G_2^{p(x)}\in L^{t_2}(B)$ for some $\max\{1,\frac N{p_-^{4R}}\}<q_i,t_2\le\infty$ $i=0,2$, $\max\{1,\frac N{p^4_--1}\}<q_1\le\infty$ and $G_1\in L^{\infty}(B)$  and they are nonnegative,
	there holds that
\begin{equation}\label{weakhar-elliptic-t0}\inf_{B_{1}}u\ge C\Big(\fint_{B_{2}}u^{t_0}\Big)^{1/t_0}.
\end{equation}

The constant $C$ depends only on $s,p_+^4,p_-^4,q_i,t_2,t,\alpha,\|H_i\|_{L^{q_i}(B_4)}, $i=0,1,2$,
\|G_2^{p(x)}\|_{L^{t_2}(B_4)}, \linebreak
 \|G_1^{p(x)}\|_{L^{\infty}(B_4)}$, $\big(\fint_{B_4}u^{sq'}\big)^{\frac{p_+^4-p_-^4}{sq'}}$, $\big(\fint_{B_4}u^{sr_i}\big)^{\frac{p_+^4-p_-^4}{sr_i}}$ $i=0,1,2$ and $\big(\fint_{B_4}u^{ss_2}\big)^{\frac{p_+^4-p_-^4}{ss_2}}$ for certain $q'=\frac q{q-1}$, $r_i\in(1,\infty)$ such that $\frac1{q_i}+\frac1q+\frac1{r_i}=1$, $s_2\in(1,\infty)$ such that $\frac1{t_2}+\frac1q+\frac1{s_2}=1$. Here $s\ge p_+^4-p_-^4$ is arbitrary.

\end{lemm}

\begin{proof} We proceed as in the proof of Lemma \ref{superhar} by using \eqref{cacciopoli-elliptic} and the ideas in Lemma \ref{subhar-elliptic}. Recall that in this process we have $\gamma\le -(p_-^4-1)$.

In this way we get \eqref{beta2}. As in Lemma \ref{superhar}, in order to finish the proof we need to find $t_0>0$ such that \eqref{reverse} holds for $u$.
So we bound, by using \eqref{cacciopoli-elliptic}, for an arbitrary $0<r\le 2$, 
$\eta\in C_0^\infty(B_{2r})$ with $0\le \eta\le 1$, $\eta\equiv1$ in $B_r$, $|\nabla\eta|\le \frac C r$  and $\gamma=1-p_-^{2r}$,
\[\begin{aligned}
\fint_{B_r} |\nabla \log u|^{p_-^{2r}}&=\fint_{B_r}u^{-p_-^{2r}}|\nabla u|^{p_-^{2r}}\le
C\fint_{B_{2r}}u^{-p_-^{2r}}\eta^{p_+^{2r}}|\nabla u|^{p_-^{2r}}
\le C\fint_{B_{2r}}u^{-p_-^{2r}}\eta^{p_+^{2r}}\\&+
\frac C{(p_-^{2r}-1)}\fint (H_0+H_2)u^{p(x)-p_-^{2r}}\eta^{p_+^{2r}}
+\frac C{(p_-^{2r}-1)}\fint H_1 u^{p(x)-p_-^{2r}}\eta^{p_+^{2r}-1}|\nabla\eta|\\
&+\frac C{(p_-^{2r}-1)^{p_+^{2r}}}\fint G_1^{p(x)}u^{p(x)-p_-^{2r}}\eta^{p_+^{2r}-p(x)}|\nabla\eta|^{p(x)}
+\frac C{(p_-^{2r}-1)^{p_+^{2r}}}\fint G_2^{p(x)}u^{p(x)-p_-^{2r}}\eta^{p_+^{2r}}.
\end{aligned}
\]

So that,
\[\begin{aligned}
\fint_{B_r} &|\nabla \log u|^{p_-^{2r}}\le C\Big[1+\|H_0\|_{L^{q_0}(B_4)}r^{-N/q_0}\Big(\fint_{B_r}u^{q_0'(p_+-p_-)}\Big)^{1/q_0'}\\
&+\|H_2\|_{L^{q_2}(B_4)}r^{-N/q_2}\Big(\fint_{B_r}u^{q_2'(p_+-p_-)}\Big)^{1/q_2'}
+\|H_1\|_{L^{q_1}(B_4)}r^{-(1+\frac N{q_1})}\Big(\fint_{B_r}u^{q_1'(p_+-p_-)}\Big)^{1/q_1'}\\
&+\|G_1^{p(x)}\|_{L^\infty(B_4)}r^{-{p_+^{2r}}}\Big(\fint_{B_r}u^{q'(p_+-p_-)}\Big)^{1/q'}
+\|G_2^{p(x)}\|_{L^{t_2}(B_4)}r^{-N/t_2}\Big(\fint_{B_r}u^{t_2'(p_+-p_-)}\Big)^{1/t_2'}\Big].
\end{aligned}\]

Now, since $q_i'<r_i$, $t_2'< s_2$,
\[\begin{aligned}
\fint_{B_r} |\nabla \log u|^{p_-^{2r}}&\le C\Big[1+\|H_0\|_{L^{q_0}(B_4)}r^{-N/q_0}M_2^{p_+-p_-}
+\|H_2\|_{L^{q_2}(B_4)}r^{-N/q_2}M_4^{p_+-p_-}\\
&+\|H_1\|_{L^{q_1}(B_4)}r^{-(1+\frac N{q_1})}M_3^{p_+-p_-}
+\|G_1^{p(x)}\|_{L^\infty(B_4)}r^{-{p_+^{2r}}}M_1^{p_+-p_-}\\
&+\|G_2^{p(x)}\|_{L^{t_2}(B_4)}r^{-N/t_2}M_5^{p_+-p_-}\Big].
\end{aligned}\]

Finally, since $0<r\le 2$, $\frac N{q_i}<p_-^4$, $i=0,2$,  $\frac N{t_2}<p_-^4$, $1+\frac N{q_1}\le p_-^4$ and $p_-^4\le p_-^{2r}\le p_+^{2r}$,
\[\begin{aligned}
\fint_{B_r} |\nabla \log u|^{p_-^{2r}}&\le C\Big[1+\sum_{i=0}^3\|H_i\|_{L^{q_i}(B_4)}+\|G_1^{p(x)}\|_{L^\infty(B_4)} +\|G_2^{p(x)}\|_{L^{t_2}(B_4)} \Big]r^{-p_+^{2r}}M^{p_+^{4}-p_-^{4}}.
\end{aligned}\]

Now the proof follows in a standard way as in Lemma \ref{superhar}
\end{proof}

\begin{rema}[Improved weak Harnack] With the same proof as that of Lemma \ref{superhar-improved} we can improve on Lemma \ref{superhar-elliptic}. In fact, \eqref{weakhar-elliptic-t0} holds for any $t_0>0$ if $p_-^4\ge N$ and for any $0<t_0<\frac N{N-p_-^4}(p_-^4-1)$ if $N>p_-^4$.
\end{rema}

\begin{rema}[Local bounds] \label{local bounds-elliptic} As in the previous section,  by modifying the proof of Lemmas \ref{lemma-cacciopoli-elliptic} and \ref{subhar-elliptic}, we get that if $u$ satisfies weakly
 \[
 \big|\mbox{div}A(x,u,\nabla u)\big|\le H_2(x)(|u|+1)^{p(x)-1}+G_2(x)|\nabla u|^{p(x)-1}\quad\mbox{in}\quad\Omega
 \]
 and
 \begin{enumerate}
\item $A(x,u(x),\nabla u(x))\cdot\nabla u(x)\ge \alpha|\nabla u(x)|^{p(x)}-H_0(x) (|u(x)|+1)^{p(x)}$ in $\Omega$

	\item $\big|A(x,u(x),\nabla u(x))\big|\le H_1(x)(|u|+1)^{p(x)-1}+G_1(x)|\nabla u|^{p(x)-1}$ in $\Omega$
	\end{enumerate}
with $0\le H_i\in L^{q_i(x)}(\Omega)$, $0\le G_1\in L^\infty(\Omega)$, $0\le G_2^{p(x)}\in L^{t_2(x)}(\Omega)$ with $q_i,t_2\in C(\Omega)$ and $\max\{1,\frac N{p(x)}\}<q_2(x),t_2(x)$ in $\Omega$, $\max\{1,\frac N{p(x)-1}\}<q_0(x),q_1(x)$ in $\Omega$, there holds that $u$ is locally bounded.

\smallskip

Then, as in the proof of Corollary \ref{coro-holder}  we get that, if the structure conditions (1), (2), (3) do not depend on $M_0$,  weak solutions to \eqref{elliptic} are locally bounded. In fact, we let $u$ be a weak solution to \eqref{elliptic} and
%
\[\begin{aligned}
&H_i(x)=g_i(x)+C_i(x),\quad i=0,1\\
&H_2(x)=f(x)+C_2(x)\\
&G_j(x)=K_j(x), \quad j=1,2
\end{aligned}
\]

Then,
\[
\big|\mbox{div} A(x, u,\nabla u)\big|=\big| B(x, u,\nabla  u)\big|\le H_2(x)(|u(x)|+1)^{p(x)-1}+G_2(x)|\nabla u(x)|^{p(x)-1}\]
and
 \begin{enumerate}
\item $ A(x, u(x),\nabla  u(x))\cdot\nabla  u(x)\ge \alpha|\nabla  u(x)|^{p(x)}-H_0(x)  (|u(x)|+1)^{p(x)}$ in $\Omega$

	\item $\big| A(x, u(x),\nabla  u(x))\big|\le H_1(x)(|u|+1)^{p(x)-1}+G_1(x)|\nabla  u|^{p(x)-1}$ in $\Omega$.
	\end{enumerate}
So, we get that $ u$ is locally bounded.
\end{rema}

\bigskip

We can now prove Harnack inequality for solutions of general elliptic equations with non-standard growth.

\begin{theo}\label{harnack-elliptic}
Let $\Omega\subset\R^N$ be bounded and let $p$ log-H\"older continuous in $\Omega$.
Let $A(x,s,\xi)$, $ B(x,s,\xi)$ satisfy the structure conditions (1), (2) and (3) for certain nonnegative functions $g_0,C_0\in L^{q_0}(\Omega)$, $g_1,C_1\in L^{q_1}(\Omega)$, $f,C_2\in L^{q_2}(\Omega)$, $K_1\in L^\infty(\Omega)$, $K_2^{p(x)}\in L^{t_2}(\Omega)$ with $\max\{1,\frac N{p_1-1}\}<q_0,q_1\le\infty$, $\max\{1,\frac N{p_1}\}<q_2,t_2\le\infty$.

 Let $\Omega'\subset\subset\Omega$. There exists $R\le \min\{1,\frac14\mathrm{dist}(\Omega',\partial\Omega)\}$  such that, if $u\ge0$ is a  bounded weak solution to \eqref{elliptic} in $\Omega$,  there exists and $C>0$ such  that, for every $x_0\in\Omega'$,
\begin{equation}\label{eq-harnack-elliptic-alt}
\sup_{B_{R}(x_0)}u\le C\big[ \inf_{B_{R}(x_0)}u+ {R+\mu R}\big].
\end{equation}
Here
\[
\begin{aligned}
\mu&=\Big[R^{1-\frac N{q_2}}\|f\|_{L^{q_2}(B_{4R})}\Big]^{\frac1{p_-^{4R}-1}}+\Big[R^{-\frac N{q_0}}\|g_0\|_{L^{q_0}(B_{4R})}\Big]^{\frac1{p_-^{4R}-1}}
+\Big[R^{-\frac N{q_1}}\|g_1\|_{L^{q_1}(B_{4R})}\Big]^{\frac1{p_-^{4R}-1}}
\end{aligned}
\]

The constant $C$ depends only on $s,p_+^{4R},p_-^{4R},q_i,t_2,\alpha$, $\mu^{p_+^{4R}-p_-^{4R}}$, $\|C_i\|_{L^{q_i}(B_{4R}(x_0))}$, $i=0,1,2$,
$\|K_2^{p(x)}\|_{L^{t_2}(B_{4R}(x_0))}$,
$ \|K_1^{p(x)}\|_{L^{\infty}(B_{4R}(x_0))}$, $\|u\|_{L^{sq'}(B_{4R}(x_0))}^{p_+^{4R}-p_-^{4R}}$, $\|u\|_{L^{sr_i}(B_{4R}(x_0))}^{p_+^{4R}-p_-^{4R}}$ $i=0,1,2$, $\|u\|_{L^{ss_2}(B_{4R}(x_0))}^{p_+^{4R}-p_-^{4R}}$ for certain $q'=\frac q{q-1}$, $r_i\in(1,\infty)$ such that $\frac1{q_i}+\frac1q+\frac1{r_i}=1$, $s_2\in(1,\infty)$ such that $\frac1{t_2}+\frac1q+\frac1{s_2}=1$. Here $s\ge p_+^4-p_-^4$ is arbitrary.

Observe that $\mu^{p_+^{4R}-p_-^{4R}}$ is bounded independently of $R$.
\end{theo}
\begin{proof} Without loss of generality we will assume that $x_0=0$. Let us call
\[
\begin{aligned}
H_0(x)&= \frac{g_0(Rx)}{R^{-\frac N{q_0}}\|g_0\|_{L^{q_0}(B_{4R})}}+R^{p(Rx)-1}{C_0(Rx)}\\
H_1(x)&=\frac{g_1(Rx)}{R^{-\frac N{q_1}}\|g_1\|_{L^{q_1}(B_{4R})}}+R^{p(Rx)-1}{C_1(Rx)}\\
H_2(x)&=\frac{f(Rx)}{R^{-\frac N{q_2}}\|f\|_{L^{q_2}(B_{4R})}}+R^{p(Rx)}{C_2(Rx)}\\
G_1(x)&=K_1(Rx)\\
G_2(x)&=R{K_2(Rx)}
\end{aligned}
\]

Let
\[
\begin{aligned}
\bar u(x)&=1+\mu+\frac{u(Rx)}R,\qquad \bar p(x)=p(Rx).
\end{aligned}
\]

If a function is identically zero in $B_{4R}(x_0)$ the corresponding term does not appear in the definition of the functions $H_i$.

Then, $\|G_1(x)^{\bar p(x)}\|_{L^{\infty}(B_4)}\le \|K_1(x)^{ p(x)}\|_{L^{\infty}(B_{4R})}$ and, for $i=0,1$,
 \[
\begin{aligned}
\Big(\fint_{B_4}H_i^{q_i}\Big)^{1/q_i}&\le C_{N,q_i}\Big[1+\Big(\int_{B_{4R}}R^{q_i(p(x)-1)- N}C_i^{q_i}\Big)^{1/q_i}\Big]\\
&\le C_{N,q_i}\big[1+\|C_i\|_{L^{q_i}(B_{4R})}\big]
 \end{aligned}
 \]
 since $q_i>\frac N{p_1-1}$ for $i=0,1$ and $0<R\le 1$.

 On the other hand, since $q_2>\frac N{p_1}$,
  \[
\begin{aligned}
\Big(\fint_{B_4}H_2^{q_2}\Big)^{1/q_2}&\le C_{N,q_2}\Big[1+\Big(\int_{B_{4R}}R^{{q_2}p(x)- N}C_2^{q_2}\Big)^{1/q_2}\Big]\\
&\le C_{N,q_2}\big[1+\|C_2\|_{L^{q_2}(B_{4R})}\big]
 \end{aligned}
 \]
 and, since $t_2>\frac N{p_1}$,
\[
 \begin{aligned}
\Big(\fint_{B_4}G_2^{t_2\bar p(x)}\Big)^{1/t_2}&\le C_{N,t_2}\Big[1+\Big(\int_{B_{4R}}R^{{t_2}p(x)- N}K_2^{t_2p(x)}\Big)^{1/t_2}\Big]\\
&\le C_{N,q_2}\big[1+\|K_2^{p(x)}\|_{L^{t_2}(B_{4R})}\big]
 \end{aligned}
 \]

 On the other hand, for $0<R\le 1$ let
\[
\bar A(x,s,\xi):=A\big(Rx,{R(s-1-\mu)}, \xi\big)
\]

Then, $\bar A\big(x,\bar u(x),\nabla\bar u(x)\big)=A\big(Rx,u(Rx),\nabla u(Rx)\big)$
and we have,
\[\begin{aligned}
\Big|\mbox{div}\bar A\big(x,&\bar u(x),\nabla\bar u(x)\big)\Big|\le
Rf(Rx)+RC_2(Rx)u(Rx)^{p(Rx)-1}+RK_2(Rx)|\nabla u(Rx)|^{p(Rx)-1}\\
&\le H_2(x)\bar u(x)^{\bar p(x)-1}+G_2(x)|\nabla\bar u(x)|^{\bar p(x)-1}.
\end{aligned}
\]

Also,
\[
\begin{aligned}
\Big|\bar A\big(x,\bar u(x),\nabla\bar u(x)\big)\Big|&\le g_1(Rx)+C_1(Rx)u(Rx)^{p(Rx)-1}+K_1(Rx)|\nabla u(Rx)|^{p(Rx)-1}\\
&\le  H_1(x)\bar u(x)^{\bar p(x)-1}+G_1(x)|\nabla\bar u(x)|^{\bar p(x)-1}
\end{aligned}
\]
and
\[
\begin{aligned}
\bar A\big(x,\bar u(x),\nabla\bar u(x)\big)\cdot\nabla \bar u(x)&\ge
\alpha|\nabla u(Rx)|^{p(Rx)}-C_0(Rx)u(Rx)^{p(Rx)-1}-g_0(Rx)\\
&\ge \alpha|\nabla\bar u(x)|^{\bar p(x)}-H_0(x)\bar u(x)^{\bar p(x)-1}.
\end{aligned}
\]

Thus, since $\bar u\ge 1$ and
 \[\begin{aligned}
 \|\bar u\|_{L^t(B_4)}^{\bar p_+^4-\bar p_-^4}&\le C\big[1+\mu^{ p_+^{4R}- p_-^{4R}}+R^{-\frac Nt(p_+^{4R}-p_-^{4R})}\|u\|_{L^t(B_{4R})}^{p_+^{4R}-p_-^{4R}}\big]\\
 &\le C\big[1+\mu^{ p_+^{4R}- p_-^{4R}}+\|u\|_{L^t(B_{4R})}^{p_+^{4R}-p_-^{4R}}\big],
 \end{aligned}
 \]
 by applying  Lemmas \ref{subhar-elliptic} and \ref{superhar-elliptic} to $\bar u$ we get the result.
\end{proof}

\smallskip

\begin{rema} Since $p$ is continuous in $\overline\Omega$ we can choose $R$ small enough in such a way that, by choosing $s$ small enough, $M_j^{p_+^{4R}-p_-^{4R}}\le \big(\fint_{B_{4R}(x_0)}u^{p_1}\big)^{\frac{p_+^{4R}-p_-^{4R}}{p_1}}\le c\big(1+\big(\int_\Omega u^{p(x)}\big)^{\frac{p_2}{p_1}-1}\big)$, $j=1,\cdots,5$ where $p_1=\inf_\Omega p$, $p_2=\sup_\Omega p$ and the constant $c$ depends only on the log H\"older modulus of continuity of $p$ in $\Omega$..

So that, if moreover, the constant $\alpha$ and the functions  $g_0,g_1,f,C_0,C_1,C_2, K_1$ and $K_2$ in the structure conditions do not depend on $M_0$, Harnack inequality holds --on small enough balls depending only on $p$-- for any nonnegative   weak solution, with a constant $C$ depending on $u$ only through $\big(\int_\Omega u^{p(x)}\big)^{\frac{p_2}{p_1}-1}$.

\end{rema}

\smallskip

From Harnack inequality we get H\"older continuity. There holds,
\begin{coro} \label{smoothness-elliptic} Let $\Omega\subset\R^N$  bounded. Let $p$  log-H\"older continuous in $\Omega$ and $p_1=\inf_\Omega p(x)$.
Let $A(x,s,\xi),\ B(x,s,\xi)$
satisfy the structure conditions (1), (2), (3) at the beginning of the section.  Assume that $g_0,C_0\in L^{q_0}(\Omega)$, $g_1,C_1\in L^{q_1}(\Omega)$ and $\max\{1,\frac N{p_1-1}\}<q_0,q_1\le\infty$,
 $f,  C_2\in L^{q_2}(\Omega)$, $K_2^{p(x)}\in L^{t_2}(\Omega)$ and $\max\{1,\frac N{p_1}\}<q_2,t_2\le\infty$. Finally, assume $K_1\in L^\infty(\Omega)$.

Then, there holds that any bounded weak solution to \eqref{elliptic} is locally H\"older continuous in  $\Omega$.

If the functions in the structure conditions are independent of $M_0$, any weak solution is locally H\"older continuous and the constant and H\"older exponent are independent of  the $L^\infty$ bound.
 \end{coro}
\begin{proof} Under these assumptions,  for every $M_0>0$, $\Omega'\subset\subset\Omega$, there exist a universal constant $C$, a radius $R_0>0$ and $\delta>0$ such that for every $0<R\le R_0$,    $x_0\in\Omega'$ and  any  weak solution $0\le v\le M_0$,
\begin{equation}\label{har-otro}
\sup_{B_{R}(x_0)} v\le C\big[\inf_{B_{R}(x_0)} v+ R^\delta\big].
\end{equation}

In fact, we apply \eqref{eq-harnack-elliptic-alt} and observe that we are assuming that $q_0,q_1>\frac N{p_1-1}$. So that, $1-\frac N{q_0}\frac1{p_-^{4R}-1}\ge1-\frac N{q_0}\frac1{p_1-1}:=\delta_0>0$,
$1-\frac N{q_1}\frac1{p_-^{4R}-1}\ge1-\frac N{q_1}\frac1{p_1-1}:=\delta_1>0$. On the other hand,
if $q_2\ge N$,
$1+\big(1-\frac N{q_2}\big)\frac1{p_-^{4R}-1}\ge1+\big(1-\frac N{q_2}\big)\frac1{p_2-1}:=\delta_2\ge1$, if $\frac N{p_1}<q_2<N$, $1+\big(1-\frac N{q_2}\big)\frac1{p_-^{4R}-1}\ge1+\big(1-\frac N{q_2}\big)\frac1{p_1-1}:=\bar \delta_2>0$ .

\smallskip

Once we have \eqref{har-otro}, we deduce that $u$ is H\"older continuous in a standard way by applying \eqref{har-otro} with $R=R_02^{-(j+1)}$ to $v_1(x)=\sup_{B_{R_02^{-j}(x_0)}}u-u(x)$ and to $v_2(x)=u(x)-\inf_{B_{R_02^{-j}(x_0)}}u$. Here,  $M_0=\sup_\Omega u$ (see \cite{GT} for the details).

\smallskip

Recall that, when the functions in the structure condition are independent of $M_0$, any weak solution is locally bounded. So that, they are locally H\"older continuous and the H\"older exponent and constant are independent of the $L^\infty$ bounds.
\end{proof}

\bigskip

Now, we assume that $A$ and $B$ satisfy the following structure conditions:
For every  $M_0>0$ there exist a constant $\alpha$ and nonnegative functions $f, g_0,g_1, C_0, C_1,  C_2, K_1,K_2$ as before and   $b\in \R_{>0}$   such that, for every $x\in\Omega$, $|s|\le M_0$, $\xi\in\R^N$,
\begin{enumerate}
\item $A(x,s,\xi)\cdot\xi\ge \alpha|\xi|^{p(x)}-C_0|s|^{p(x)}-g_0(x)$.

\item $\big|A(x,s,\xi)\big|\le g_1(x)+C_1|s|^{p(x)-1}+K_1|\xi|^{p(x)-1}$.

\item[(3')] $\big|B(x,s,\xi)\big|\le f(x)+C_2|s|^{p(x)-1}+K_2|\xi|^{p(x)-1}+b|\xi|^{p(x)}$

\end{enumerate}

We will prove Harnack inequality for bounded weak solutions.

In fact, for  $0\le u\le M_0$  we can reduce the problem to the case of $b=0$ treated before since, on one hand, there holds that,
\[
\begin{aligned}
&\mbox{div}A(x,u,\nabla u)\ge -  \big(f(x)+C_2(x)u^{p(x)-1}+K_2(x)|\nabla u|^{p(x)-1}+b|\nabla u|^{p(x)}\big)\quad\mbox{in }B_r\\
&\Rightarrow\\
&\mbox{div}\widetilde A(x,u,\nabla u)\ge  - \big(f(x)+C_2(x)u^{p(x)-1}+K_2(x)|\nabla u|^{p(x)-1}\big)\quad\mbox{in }B_r
\end{aligned}
\]
with $\widetilde A(x,s,\xi)= e^{\frac b\alpha(s-M_0)}A(x,s,\xi)$ satisfying,
\begin{enumerate}
\item $\widetilde A(x,u(x),\nabla u(x))\cdot\nabla u(x)\ge \alpha e^{-\frac b\alpha M_0}|\nabla u(x)|^{p(x)}-C_0(x)u(x)^{p(x)}-g_0(x)$.

\item $\big|\widetilde A(x,u(x),\nabla u(x))\big|\le g_1(x)+C_1(x)|u(x)|^{p(x)-1}+K_1(x)|\nabla u(x)|^{p(x)-1}$.
    \end{enumerate}

\medskip

On the other hand,  again for $0\le u\le M_0$ there holds that,
\[
\begin{aligned}
&\mbox{div}A(x,u,\nabla u)\le f(x)+C_2(x)u^{p(x)-1}+K_2(x)|\nabla u|^{p(x)-1}+b|\nabla u|^{p(x)}\quad\mbox{in }B_r\\
&\Rightarrow\\
&\mbox{div}\bar A(x,u,\nabla u)\le e^{\frac b\alpha M_0}\big(f(x)+C_2(x)u^{p(x)-1}+K_2(x)|\nabla u|^{p(x)-1}\big)\quad\mbox{in }B_r
\end{aligned}
\]
with $\bar A(x,s,\xi)= e^{\frac b\alpha(M_0-s)}A(x,s,\xi)$ satisfying,
\begin{enumerate}
\item $\bar A(x,u(x),\nabla u(x))\cdot\nabla u(x)\ge \alpha |\nabla u(x)|^{p(x)}-e^{\frac b\alpha M_0}\big(C_0(x)u(x)^{p(x)}+g_0(x)\big)$.

\item $\big|\bar A(x,u(x),\nabla u(x))\big|\le e^{\frac b\alpha M_0}\big( g_1(x)+C_1(x)|u(x)|^{p(x)-1}+K_1(x)|\nabla u(x)|^{p(x)-1}$.
    \end{enumerate}

\medskip

Thus, there holds
\begin{theo}\label{harnack-elliptic-alt} Let $\Omega\subset\R^N$ be bounded and let $p$ log-H\"older continuous in $\Omega$.
Let $A(x,s,\xi)$, $B(x,s,\xi)$ satisfy the structure conditions (1), (2), (3').  Let $u\ge0$ be a  bounded weak solution to \eqref{elliptic} and let $M_0$ be such that $u\le M_0$ in $\Omega$. Let $\Omega'\subset\subset\Omega$. There exists  $R_0\le \min\{1,\frac14\mathrm{dist}(\Omega',\partial\Omega)\}$ such that if $x_0\in\Omega'$
and $0<R\le R_0$,
\begin{equation}\label{eq-harnack-elliptic}
\sup_{B_{R}(x_0)}u\le C\big[ \inf_{B_{R}(x_0)}u+{R+\mu R}\big]
\end{equation}
where
\[
\begin{aligned}
\mu&=\Big[R^{1-\frac N{q_0}}\|f\|_{L^{q_2}(B_{4R})}\Big]^{\frac1{p_-^{4R}-1}}+\Big[R^{-\frac N{q_0}}\|g_0\|_{L^{q_0}(B_{4R})}\Big]^{\frac1{p_-^{4R}-1}}
+\Big[R^{-\frac N{q_1}}\|g_1\|_{L^{q_1}(B_{4R})}\Big]^{\frac1{p_-^{4R}-1}},\\
\end{aligned}
\]

The constant $C$ depends only on  $bM_0,\alpha, s,q_i$, $i=0,1,2$,  the log H\"older modulus of continuity of $p$ in $\Omega$, $\mu^{p_+^{4R}-p_-^{4R}}$, and  $M^{p_+^{4R}-p_-^{4R}}$ where $p_+=\sup_{B_{4R}(x_0)}  p$, $p_-=\inf_{B_{4R}(x_0)}  p$, $\|K_1^{p(x)}\|_{L^{\infty}(B_{4R}(x_0))}$, $\|K_2^{p(x)}\|_{L^{t_2}(B_{4R}(x_0))}$, $M=\sum_{j=1}^4 M_j$ and $M_{1}=\big(\fint_{B_{4R}(x_0)}u^{sq'}\big)^{1/sq'}$, $M_{i+2}=\big(\fint_{B_{4R}(x_0)}u^{sr_i}\big)^{1/sr_i}$  for certain $q'=\frac q{q-1}$ depending on $q_i$, $p_1$ and $N$  and $r_i\in(1,\infty)$ $i=0,1,2$ with $\frac1{q_i}+\frac1q+\frac1{r_i}=1$. Here $s\ge p_+-p_-$ is  arbitrary.

Observe that $\mu^{p_+^{4R}-p_-^{4R}}$ and $M^{p_+^{4R}-p_-^{4R}}$ are bounded independently of $R$.

\end{theo}
\begin{proof}
Theorem \ref{harnack-elliptic-alt} is obtained from Lemmas \ref{subhar-elliptic} and \ref{superhar-elliptic} applied to $\bar u$ with the operator $A$ replaced by $\widetilde A$ and $\bar A$ respectively.
\end{proof}

\medskip

With the same proof as that of Corollary \ref{smoothness-elliptic} get the following regularity result.
\begin{coro} \label{smoothness-elliptic-alt} Let $\Omega\subset\R^N$ bounded. Let $A$ and $B$ satisfy the structure conditions (1),(2), (3'). Let $u$ be a bounded weak solution to \eqref{elliptic} in $\Omega$ with $p$ log-H\"older continuous.  Then, $u$ is locally H\"older continuous in $\Omega$.
\end{coro}

\begin{rema} Observe that, under condition (3') the constant in Harnack inequality and the H\"older exponent and constant of a bounded weak solution depend explicitly on the $L^\infty$ bound.
\end{rema}

\end{section}

\bigskip

\begin{section}{Strong maximum principle for $p(x)-$superharmonic functions.}
\label{sect-strong-maximum-principle}

In this section we prove the strong maximum principle for $p(x)-$superharmonic functions. As stated at the introduction, the strong maximum principle cannot be
deduced from Harnack inequality as in que case $p$ constant. Instead, we will use some barriers constructed in \cite{FBMW}.
\begin{prop}[Lemma B.4 in \cite{FBMW}] \label{prop-barrier} Suppose that $p(x)$ is Lipschitz continuous. Let
$w_{\mu}=Me^{-\mu|x|^2}$, for $M>0$ and $r_1\ge|x|\ge r_2>0$. Then, there
exist ${\mu}_0,\ep_0>0$ such that, if $\mu>{\mu}_0$ and $\|\nabla
p\|_{\infty}\leq \ep_0$,
\begin{align*}&\mu^{-1}e^{\mu|x|^2}M^{-1}|\nabla w|^{2-p}\Delta_{p(x)}w_{\mu}\geq C_1 (\mu-C_2 \|\nabla
p\|_{\infty} |\log M|)
 \mbox{ in }B_{r_1}\setminus B_{r_2}.\end{align*}
Here $C_1,C_2$ depend only on $r_2,r_1, p_+,p_-$,\ \ \
${\mu}_0={\mu}_0(p_+,p_-,N,\|\nabla p\|_{\infty},r_2,r_1)$
and \newline
$\ep_0=\ep_0(p_+,p_-,r_1,r_2)$.
\end{prop}

Then, we have
\begin{coro} \label{barrier}Suppose that $p(x)$ is Lipschitz continuous.  Let $A_0>0$. Then,  there exists $\delta_0>0$ depending on $p_+,p_-,\|\nabla p\|_\infty$ and $A_0$ and for every $0<A\le A_0$ there exists $\mu_0>0$ depending on the same constants and also on $A$ such that, if moreover  $\delta\le \delta_0$ and $\mu\ge\mu_0$, the function
\[
w(x)=A\ \frac{e^{-\mu\frac{|x-x_0|^2}{\delta^2}}-e^{-\mu}}{e^{-\frac\mu{4}}-e^{-\mu}}
\]
satisfies
\[\begin{cases}
\Delta_{p(x)}w\ge0\quad&\mbox{in}\quad B_\delta(x_0)\setminus B_{\delta/2}(x_0),\\
w=0 \quad&\mbox{on}\quad \partial B_\delta(x_0),\\
w=A \quad&\mbox{on}\quad \partial B_{\delta/2}(x_0).
\end{cases}\]
\end{coro}
\begin{proof}
Set $\bar w(x)=\frac 1\delta w(x_0+\delta x)$, $\bar p(x)=p(x_0+\delta x)$. Let
$M=\frac A{e^{-\frac\mu{4}}-e^{-\mu}}$. Then,
\[
\bar w(x)=M\, e^{-\mu|x|^2}+ c,\qquad|\nabla \bar p(x)|=\delta|\nabla p(x_0+\delta x)|.
\]
Hence, by Proposition \ref{prop-barrier}, if $\delta$ is small  and $\mu$ is large depending only on $p_+,p_-$ and $\|\nabla p\|_\infty$,
\[
\mu^{-1}e^{\mu|x|^2}M^{-1}|\nabla \bar w|^{2-\bar p}\Delta_{\bar p(x)}\bar w(x)\ge
 C_1 (\mu-C_2 \|\nabla \bar p\|_{\infty} |\log M|)  \mbox{ in }B_{1}\setminus B_{1/2}.\]

 Observe that $M=Ae^{\mu/4}\frac1{1-e^{-3\mu/4}}$. Therefore, if $\mu$ is large there holds that
 \[
 1\le M\le 4A e^{\mu/4},
 \]
 so that,
 \[
 |\log M|\le A\mu.
 \]
 Hence, in this situation,
 \[
\mu^{-1}e^{\mu|x|^2}M^{-1}|\nabla \bar w|^{2-\bar p}\Delta_{\bar p(x)}\bar w(x)\ge
 C_1 (1-C_2 \delta\|\nabla  p\|_{\infty} A)\mu \ge0 \mbox{ in }B_{1}\setminus B_{1/2}\]
if, moreover, $\delta$ is small depending on $C_1,C_2,A_0$ and $\|\nabla p\|_\infty$.
\end{proof}

We can now prove our main result in this section. We follow the ideas of the proof in \cite{V} for the case $p$ constant.
\begin{theo}\label{strong} Suppose that $p(x)$ is Lipschitz continuous. Let $\Omega\subset\R^N$ be connected and $0\le u\in C^1(\Omega)$ such that $\Delta_{p(x)} u\le 0$ in $\Omega$. Then, either $u\equiv0$ in $\Omega$ or $u>0$ in $\Omega$.
\end{theo}
\begin{proof} Assume the result is not true. Then, since $\Omega$ is connected, $\partial\{u>0\}\cap\Omega\neq\emptyset$. Let $x_1\in\{u>0\}$ such that $\mbox{dist\,}(x_1,\partial\{u>0\})<\mbox{dist\,}(x_1,\partial\Omega)$ and let
$y\in\partial\{u>0\}\cap\Omega$ such that $r=|x_1-y|=\mbox{dist\,}(x_1,\partial\{u>0\})$.
Let $A_0=\sup_{B_r(x_1)}u$. Let $\delta_0$ the constant in Corollary \ref{barrier}.
By choosing $x_0$ on the line between $x_1$ and $y$ and taking $\delta=|x_0-y|$ we
may assume that $\delta\le \delta_0$ and $B_\delta(x_0)\subset\{u>0\}$. Let now $A=\inf_{\partial B_{\delta/2}(x_0)}u$. Then, $0<A\le A_0$. Therefore, by taking $w$ as in Corollary \ref{barrier} we have
\[
u(x)\ge w(x)\ge0\quad\mbox{in}\quad B_{\delta}(x_0)\setminus B_{\delta/2}(x_0).
\]
Since $u(y)=w(y)=0$. There holds that,
\[
|\nabla u(y)|\ge |\nabla w(y)|>0.
\]
But this is a contradiction since $y\in\partial\{u>0\}\cap\Omega$, $u\ge 0$ in $\Omega$ and $u\in C^1(\Omega)$ so that, $\nabla u(y)=0$.
\end{proof}

\begin{rema} Recall that in \cite{AM} it was proved that solutions to $\Delta_{p(x)} u=0$ are $C^{1,\alpha}_{loc}$. Thus, Theorem~\ref{strong} applies to nonnegative weak solutions.
\end{rema}

\medskip

With a similar proof we get
\begin{theo}
Under the assumptions of Theorem \ref{strong} if, moreover, there exists $y\in\partial\Omega$ such that there is a ball $B$ contained in $\Omega$ such that $y\in\partial B$, $u\in C(\overline B)$, $u>0$ in $B$ and $u(y)=0$ then,  for $x\in B$ close enough to $y$ there holds that $u(x)\ge c_0(x-y)\cdot\nu$ where $c_0>0$ and $\nu$ the unitary direction from $y$ to the center of the ball $B$.

If moreover, $u\in C^1(\Omega\cup\{y\})$, there holds that  either $u\equiv0$ in $\Omega$ or else $\frac{\partial u(y)}{\partial \nu}>0$. Here , $\nu$ is as above.
\end{theo}

\end{section}


\begin{thebibliography}{[xxx]}

\bibitem{AM} E. Acerbi, G. Mingione, {\it Regularity results for a class of functionals with non-standard growth}, Arch. Rat. Mech. Anal. {\bf 156} (2001), 121--140.

\bibitem{AMS} {R. Aboulaich, D. Meskine, A. Souissi}, {\it New diffusion
models in image processing}, Comput. Math. Appl. {56 (4)} (2008),
874--882.

%

%
%
%


%


\bibitem{CLR} Y. Chen, S. Levine, M. Rao, {\it Variable exponent, linear
growth functionals in image restoration}, SIAM J. Appl. Math. {66
(4)} (2006), 1383--1406.

%

\bibitem{DHHR} L. Diening, P. Harjulehto, P. Hasto, M. Ruzicka, { Lebesque and Sobolev Spaces with variable exponents}, Lecture Notes in Mathematics 2017, Springer,  2011.


\bibitem{Fan} X. Fan, {\it Global $C^{1,\alpha}$ regularity for variable exponent elliptic equations in divergence form}, Jour. Diff. Eqns.
{\bf 235} (2007), 397--417.

\bibitem{FZ} X. Fan, D. Zhao, {\it A class of De Giorgi type and H\"older continuity}, Nonlinear Anal. TM\&A {\bf 36} (1999), 295--318.


%
%
\bibitem{FBMW} J. Fernandez Bonder, S. Mart\'{\i}nez, N. Wolanski, {\it A free boundary problem for the $p(x)$-Laplacian}, Nonlinear Analysis
{\bf 72}, (2010), 1078--1103.

\bibitem{GT}
D.~Gilbarg and N.~S. Trudinger, {Elliptic partial differential equations
  of second order}, Grundlehren der Mathematischen Wissenschaften [Fundamental
  Principles of Mathematical Sciences], vol. 224, Springer-Verlag, Berlin,
  1983.

\bibitem{HHKLM} P. Harjulehto, P. Hasto, M. Koskenoja, T. Lukkari, N.
Marola, {\it An obstacle problem and superharmonic functions with
non standard growth}, Nonlinear Analysis: Theory, Methods \&
Applications {67 (12)} (2007), 3424--3440.


\bibitem{HKLMP} P. Harjulehto, T. Kuusi, T. Lukkari, N. Marola, M. Parviainen, {\it Harnack's inequality for quasiminimizers with nonstandard growth conditions}, J. Math. Anal. Appl. {\bf 344} (2008), 504--520.

\bibitem{HKM} J. Heinonen, T. Kilpel\"ainen, O. Martio, Nonlinear Potential Theory of Degenerate Elliptic Equations, Oxford University Press, Oxford, 1993.

\bibitem{KR}
O. Kov\'a\v{c}ik, J. R\'akosn{\'i}k, \emph{On spaces ${L}^{p(x)}$ and
  ${W}^{k,p(x)}$}, Czechoslovak Math. J \textbf{41} (1991), 592--618.

%
%
%
%
%
%
%
%

\bibitem{MZ} J. Maly, W. Ziemer, Fine Regularity of Solutions of Elliptic Partial Differential Equations, Math. Surveys and Monographs Vol. 51, AMS, 1997.

\bibitem{R} M. Ruzicka, { Electrorheological Fluids: Modeling and
Mathematical Theory}, Springer-Verlag, Berlin, 2000.

\bibitem{S} J. Serrin, {\it Local behavior of solutions of quasilinear equations}, Acta Math. {\bf 111} (1964), 247--302.

\bibitem{V} V\'azquez, J. L., {\it A strong maximum principle for some quasilinear elliptic equations}, Appl. Math. Optim. {\bf 12},
(1984) 191--202.


\bibitem{ZL} X. Zhang, X. Liu, {\it Local boundedness and Harnack inequality of $p(x)-$Laplace equation}, J. Math. Anal. Appl. {\bf 332} (2007), 209--218.













\end{thebibliography}
\end{document}